\documentclass[11pt]{article}
\usepackage{}
\usepackage{mathrsfs}
\usepackage{amsmath}
\usepackage{amsfonts}
\usepackage{amssymb}
\usepackage{epsfig}
\usepackage{dsfont}
\usepackage{multirow}
 \usepackage{graphicx}
\usepackage{color}
\usepackage{bm}
\usepackage{subfigure}

\makeatletter % `@' now normal "letter"
\@addtoreset{equation}{section}
\makeatother  % `@' is restored as "non-letter"

\renewcommand{\Box}{\framebox{\rule{0.3em}{0.0em}}}

\newtheorem{thm}{Theorem}[section]
\newtheorem{lema}{Lemma}[section]
\newtheorem{prop}{Proposition}[section]
\newtheorem{ex}{Example}[section]
\newtheorem{rem}{Remark}[section]

\newtheorem{alg}{Algorithm}[section]
\newtheorem{assu}{Assumption}[section]

\newcommand{\setd}{{ d \kern -.15em l}}
\newcommand{\hatsetd}{ d \hat{\kern -.15em l }}

\newcommand{\x}{{ x}}

\newcommand{\K}{{\cal K}}

\newcommand{\N}{{\cal N}}
\newcommand{\CP}{{\cal P}}

\newcommand{\bgeqn}{\begin{eqnarray}}
\newcommand{\edeqn}{\end{eqnarray}}
\newcommand{\bgeq}{\begin{eqnarray*}}
\newcommand{\edeq}{\end{eqnarray*}}
\newcommand{\bec}{\begin{center}}
\newcommand{\enc}{\end{center}}
\newcommand{\R}{{\rm I\!R}}

\newcommand{\inmat}[1]{\mbox{\rm {#1}}}

\newcommand{\D}{{\cal D}}

\newcommand{\F}{{\cal F}}

\newcommand{\M}{{\cal M}}

\newcommand{\Y}{{\cal Y}}

\newcommand{\be}{\begin{equation}}
\newcommand{\ee}{\end{equation}}

\def\w{\omega}
\def\bbe{{\Bbb{E}}} %expectation

\renewcommand{\Box}{\hfill \rule{2.3mm}{2.3mm}}

\begin{document}

\title{Discrete Approximation of Two-Stage Stochastic and Distributionally Robust Linear Complementarity Problems}

\vspace{0.5cm}

\author{Xiaojun Chen\thanks{
Department of Applied Mathematics,
The Hong Kong Polytechnic University,
%Hung Hom, Kowloon,
Hong Kong,
(maxjchen@polyu.edu.hk). This author's work is supported in part by Hong Kong Research Grant Council PolyU153016/16p} \hspace{0.3in}
Hailin Sun\thanks{
School of Economics and Management, Nanjing University of Science and Technology, Nanjing, 210094, China,
 (hlsun@njust.edu.cn). This author's work is supported in part by the National Natural Science Foundation of China (Grant No. 11401308)} \hspace{0.3in}
Huifu Xu\thanks{
 School of Mathematics,
  University of Southampton,
 Southampton, SO17 1BJ, UK,
(h.xu@soton.ac.uk). The work of this author is  supported in part by EPSRC Grant EP/M003191/1}}

\maketitle

\noindent
{\bf Abstract.}
 In this paper, we propose a discretization scheme for the
 two-stage stochastic linear complementarity  problem (LCP) where the underlying random data are continuously distributed.
   Under some moderate conditions, we derive qualitative and quantitative
 convergence for the solutions obtained from solving the discretized two-stage stochastic LCP (SLCP).
   We explain how the discretized two-stage SLCP may be solved by
 the well-known progressive hedging method (PHM).
 Moreover, we extend the discussion by
considering a  two-stage distributionally robust LCP (DRLCP) with moment constraints and
proposing a discretization scheme for the DRLCP.
As an application,
we show how the SLCP and DRLCP models
can be used to study
equilibrium arising from
two-stage duopoly game where each player
plans to set up its optimal capacity at present with anticipated
competition for production in future.


\section{Introduction}
%\section{two-stage SLCP}

Let $\xi:\Omega\to\R^l$ be a random variable defined in the probability space $(\Omega, \F, P)$ with support set $\Xi\subset \R^l$ and $\Y$ be the space of measurable functions defined on $\Xi$.
We consider the following two-stage stochastic linear
complementarity problem: 
find $(x,y(\cdot))\in \R^n\times \Y$ which solves
\begin{equation}
\inmat{(SLCP)} \quad \label{eq:slcp-two1}
\left\{
\begin{array}{l}
0\leq x \perp A x + \bbe[B(\xi)y(\xi)] +  q_1\geq0, \\
0\leq y(\xi) \perp M(\xi)y(\xi) + N(\xi)x  + q_2(\xi)\geq0, \;\;  \mbox{ for a.e.}\, \xi\in\Xi,
\end{array}\right.
\end{equation}
where $A\in\R^{n\times n}$, $q_1\in\R^n$, $B(\cdot): \R^{l} \to \R^{n \times m}$,
$M(\cdot): \R^{l} \to \R^{m \times m}$, $N(\cdot): \R^{l} \to \R^{m\times n} $ and $q_2(\cdot): \R^{l} \to \R^m$
are continuous matrix valued mappings,  the mathematical expectation is taken componentwise w.r.t. probability distribution of $\xi$, and abbreviation ``a.e.'' stands for almost everywhere.
In the case when $\xi$ follows a finite discrete distribution, the problem above reduces to a deterministic complementarity problem which has been extensively investigated over the past few decades, see
monographs \cite{CoPaSt92,FaPa03}. 
Note that if  we consider $(\Xi,\mathscr{B})$ as a measurable space  equipped with Borel sigma algebra $\mathscr{B}$, then
 $P$ may be viewed as probability measure
defined on $(\Xi,\mathscr{B})$ induced by the random variable $\xi$ and consequently
$P$ is called the probability distribution of $\xi$.
Throughout the paper, we use terminology probability measure and probability distribution interchangeably.

In the first stage of SLCP (\ref{eq:slcp-two1}), one is supposed to
find $x\in \R^n$ here and now before the random data $\xi(\w)$ is available. At the second stage
 when $x$ is fixed and a realization of $\xi(\w)=\xi$ becomes known,
$y\in\R^m$ is sought to satisfy the second stage of SLCP
 (\ref{eq:slcp-two1}). The two-stage
 model arises naturally from first order optimality conditions of a two-stage stochastic
linear program with recourse (see \cite{BiL11,SP09}). It can also be used to characterize
equilibrium arising from a two-stage stochastic game where at the first stage players compete
for capacity expansion before realization of uncertainty and at the second stage they bid for producing goods or services after production capacity is developed and uncertainty (i.e. market demand) is observed.

The two-stage SLCP  (\ref{eq:slcp-two1}) may be regarded as a special case of two-stage
stochastic variational inequalities (SVI)
considered by Chen, Pong and  Wets \cite{XTW17} and multi-stage SVI
 developed by Rockafellar and Wets \cite{rw2017}, which  synthesize and extend the well investigated ERM models \cite{XF05, XWZ12} and expected value models \cite{Gurkan99a, JX08} for one-stage SVI. In a more recent
development, Rockafellar and Sun \cite{rs2017} apply the well-known 
PHM to solve multi-stage SVI and SNCP. However,  the application is restricted to the case where the distribution of  the underlying uncertainty is
discrete and finite.

In this paper, we discuss discrete approximation of the two-stage SLCP which is partly aimed to fill out
 the gap on application of  PHM to the continuously distributed SLCP. We are focusing on a
 two-stage SLCP so that we may concentrate on the key ideas in our approach and leave the potential extension to multi-stage and/or non-linear case for future research.

Extending from the two-stage SLCP (\ref{eq:slcp-two1}), we
 consider a situation where  the true probability distribution $P$ is unknown but it is possible to
 use partial information
 to construct an ambiguity set $\CP$ of distributions which contains the true distribution.
Consequently,  we
propose  a  two-stage  distributionally robust
 LCP as follows:
\begin{equation}\label{eq:slcp-two1-1-DRO}
\inmat{(DRLCP)} \quad \left\{
\begin{array}{ll}
0\leq x \perp A x + \bbe_P[B(\xi)y(\xi)] +  q_1\geq0, &\forall P\in \CP,\\
0\leq y(\xi) \perp M(\xi)y(\xi) + N(\xi)x  + q_2(\xi)\geq0, &  \inmat{ for } P\inmat{-a.e. }  \xi\in\Xi,  P\in \CP.
\end{array}
\right.
\end{equation}
The DRLCP
requires every solution to satisfy the first stage complementarity condition
for all $P\in\CP$ to mitigate the risk arising from ambiguity of the true distribution.
Obviously the new model is more demanding on its solution than the two-stage
SLCP  \eqref{eq:slcp-two1}
and as a result, it might not have a solution if the ambiguity set is too large.
Like the two-stage SLCP \eqref{eq:slcp-two1} which stems from two-stage stochastic linear programming or
two-stage stochastic games with continuous actions, the two-stage 
 DRLCP \eqref{eq:slcp-two1-1-DRO}
can be linked to distributionally robust optimization and games.
Indeed, the two-stage DRLCP \eqref{eq:slcp-two1-1-DRO} may be used to characterize
the first order optimality conditions of the so-called ex post optimal
solution to a two-stage distributionally robust optimization problem
and  the ex post equilibrium of two-stage distributionally robust games.

The paper has three main contributions.
\begin{itemize}

\item[(i)] We provide sufficient conditions for the existence and
 uniqueness of the solution  of the two-stage SLCP (Proposition \ref{p:solutiony-x}) and
propose a discrete approximation scheme for the problem.
We then present
some qualitative and quantitative convergence analysis of the solutions obtained from solving
the discretized two-stage SLCP to their true counterparts (Theorems \ref{t:convergence} and \ref{t-quantitative-sect3}).
Application of  PHM to the discretized problem is outlined (Section 3.3).

\item[(ii)]
In the absence of complete information on the true probability distribution,
we propose a distributionally robust model for the
 two-stage SLCP,
the model is
parallel to the
ex post equilibrium of robust games studied by Aghassi and Bertsimas \cite{AgB06}.
We derive a dual formulation of the DRLCP model and discuss a discretization approach
for the latter  (Section 4).

\item[(iii)]
As an application as well as a motivation,
we propose a
two-stage distributionally robust game in a duoploy market where two players need to make
strategic decisions on capacity for future production with anticipation of
Nash-Cournot type competition after demand uncertainty in future is observed.
Under some standard conditions, we reformulate the problem as
 a two-stage SLCP when the true  distribution of the demand uncertainty is known and a two-stage DRLCP otherwise.
We give an academic  example to show existence of a solution to the
 two-stage DRLCP  (Section 5.1).

\end{itemize}

Throughout the paper, we use the following notation.
 $\R_{+}^{n}$
denotes the
  non-negative orthant
 and $\R_{++}^{n}$  the interior of $\R_+^n$.
For a vector $a\in\R^n$, we write $(a)_+$ for $\max(0,a)$,
where the maximum is taken componentwise.
    For matrices $A,B\in\R^{n\times n}$,
we write
 $A\succeq B$ and $A\succ B$
 to indicate $A-B$ being positive semidefinite
 and positive definite respectively.
  Differing from the convention in semi-definite programming, here
 $A$ and $B$ are not necessarily symmetric.
  We use $\|\cdot\|$ to denote the $2$-norm in
 both vector and matrix spaces and indicate
 any other norms by a subscript such as
    the infinity norm $\|\cdot\|_\infty$.
  Finally, to ease the exposition, we write
 $i\in \bar{K}$ for $i=1,\cdots,K$.

\section{Structure of the two-stage SLCP}

Although our main emphasis, later on in this paper, will rest on the case
where the second stage of SLCP  (\ref{eq:slcp-two1}) has a unique solution, we believe that it will be helpful to discuss a precise meaning of the model  in a general setting where the second stage of SLCP  (\ref{eq:slcp-two1})
has
multiple solutions for each fixed $x$ and $\xi$.
Let $\Y(x,\xi)$ denote the set of solutions of the second stage of SLCP  (\ref{eq:slcp-two1}). Then the two-stage SLCP can be
written as: find $x\in\R^n$ and $y(\cdot)\in \Y(x,\xi(\cdot))$ which solve
\bgeqn
0\leq x \perp A x + \bbe[B(\xi)y(\xi)] +  q_1\geq0.
\label{eq:slcp-two1-2-a}
\edeqn
Here $\Y(x,\xi(\cdot)): \Omega\to \Y$ is a random set-valued mapping, and
$y(\cdot)$ is a
measurable selection  such that
$
\bbe[ B(\xi)y(\xi)]
$
is finite-valued. In the case when $\Y(x,\xi)$ is a singleton for each $x\in \R^n$ and $\xi\in\Xi$,
 $\Y(x,\xi)=\{\bar{y}(x, \xi)\}$.

In what follows, we investigate conditions for the
 existence and uniqueness of a solution to (\ref{eq:slcp-two1}). For this purpose, we
 make the following technical
assumptions.

\begin{assu}\label{A-AD}
  There exists a
 positive continuous function $\kappa(\xi)$ such that $\bbe[\kappa(\xi)]<+\infty$ and for almost every $\xi$,
\bgeqn\label{eq:ABmatrix}
\begin{pmatrix}
z^T, u^T
\end{pmatrix}
\begin{pmatrix}
A & B(\xi)\\
N(\xi) & M(\xi)
\end{pmatrix}\begin{pmatrix}
z\\
u
\end{pmatrix} \geq \kappa(\xi)(\|z\|^2+\|u\|^2), \; \;\; \forall z\in\R^n, \; u\in\R^{m}.
\edeqn
\end{assu}

Let ${\cal D}$ denote the set of $m\times m$ diagonal matrices
 $D$ with diagonal components $D_{jj}\in \{0,1\}$, for $j\in\bar{m}$, and
 $M$ be an $m\times m$ positive definite matrix. Let
 $\mathscr{J}$ denote the power set of
$\{1,\cdots,n\}$ and $J\in \mathscr{J}$. Let 
$D_J\in {\cal D}$ with
 $$
 (D_J)_{jj} = \left\{
 \begin{array}{ll}
 1, & \inmat{if} \;\; j \in J,\\
 0, & \inmat{otherwise}.
 \end{array}\right.
$$
It is known that
$I - D_J (I- M)$ is invertible when $M\succ 0$, see for instance
\cite{XC13}. Let
$$
 U_J(M)= (I - D_J (I- M))^{-1}D_J.
 $$
By permutation if necessary,
we assume
for the simplicity of exposition
that
$J=\{1, 2, \cdots |J|\}$, where $|J|$ denotes the cardinality of set $J$. 
Consequently, we know from \cite{XC13}
that
$$
U_J(M) =\left\{
\begin{array}{ccc}
0_{n\times n}, & \mbox{if} \; J = \emptyset,\\
\begin{pmatrix}
 M_J^{-1} & 0\\
 0  & 0
\end{pmatrix},
& \mbox{ otherwise},
\end{array}
\right.
$$
where $M_J $ is the $|J|\times |J|$ sub-matrix of $M$ whose entries of $M$ are indexed by the set
$J \in \mathscr{J}$.

From time to time in the follow-up discussions,  we need to
look into  positive definiteness of $A-B(\xi)U_J(M(\xi))N(\xi)$ and its inverse.
To this end, we state the following intermediate technical result.

\begin{lema}\label{l:matrixschur}
Under Assumption \ref{A-AD}, the following assertions hold.
\begin{itemize}
\item[(i)] $z^T Az\geq
\sup_{\xi\in \Xi} \kappa(\xi)\|z\|^2$
for all $z\in\R^n$,  and $u_J^T M_J(\xi)u_J\geq \kappa(\xi)\|u_J\|^2$, for all  $u_J\in\R^{|J|}$ and $J\in\mathscr{J}$.

\item[(ii)]
$
A-B(\xi) U_J(M(\xi)) N(\xi)
$
is well defined and
$
z^T(A-B(\xi) U_J(M(\xi)) N(\xi))z\geq \kappa(\xi)\|z\|^2,
$
for all $z\in\R^n$.

\item[(iii)] $\|M_J(\xi)^{-1}\|\leq \frac{1}{\kappa(\xi)}$ and
$
\|(A-B(\xi) U_J(M(\xi)) N(\xi))^{-1}\|\leq \frac{1}{\kappa(\xi)}.
$
\end{itemize}
\end{lema}
\noindent
\textbf{Proof.} We only prove Part (ii) since Part (i)
follows straightforwardly from (\ref{eq:ABmatrix})
and Part (iii) follows from Parts (i) and (ii).
By setting $u= - U_J(M(\xi))N(\xi)z$ in  \eqref{eq:ABmatrix}, and using
$U_J(M(\xi))M(\xi)U_J(M(\xi)=U_J(M(\xi))$, we have
$$
z^TAz - z^TB(\xi)U_J(M(\xi))N(\xi)z\geq \kappa(\xi)(\|z\|^2 + \|U_J(M(\xi))N(\xi)z\|^2)\geq \kappa(\xi)\|z\|^2
$$
for any $z\in \R^n$.
\hfill$\Box$

We are now ready to state
 existence of solutions to  SLCP  \eqref{eq:slcp-two1} and the
structure of the second stage solution.

\begin{prop}
\label{p:solutiony-x}
 Let Assumption \ref{A-AD} hold.
 For any given $x$ and  $\xi\in \Xi$, let $D(x,\xi) \in \D$ be an $m$-dimensional diagonal matrix with
$$
D_{jj}(x,\xi) : = \left\{
\begin{array}{ll}
1, & {\rm if } \; \big(M(\xi)y(\xi) + N(\xi)x  + q_2(\xi)\big)_j \leq y_j(\xi),\\
0, & {\rm otherwise }.
\end{array}
\right.
$$
Let
\bgeqn
W(x, \xi) = [I - D(x, \xi) (I- M(\xi))]^{-1}D(x, \xi)
\label{eq:W}
\edeqn
 and
$$
J(x, \xi) = \{j : (M(\xi)y(\xi) + N(\xi)x  + q_2(\xi))_j \leq y_j(\xi)\}.
$$
Then the following assertions hold.
\begin{itemize}

\item[(i)] 
The two-stage SLCP \eqref{eq:slcp-two1} has a unique solution
$(x^*, y^*(\cdot))\in\R^n\times \Y$.

\item[(ii)] The solution to the second stage of SLCP \eqref{eq:slcp-two1} can be written as
\begin{equation}\label{eq:soly-o}
\bar{y}(x,\xi) =- W(x, \xi)(N(\xi)x + q_2(\xi))
\end{equation}
and $\bar{y}(\cdot, \xi)$ is globally  Lipschitz continuous w.r.t $x$.

\item[(iii)] The  first equation 
of SLCP \eqref{eq:slcp-two1} 
can be reformulated as 
\begin{equation}\label{eq:slcp2_1}
0\leq x \; \perp \; (A - \bbe[B(\xi) W(x,\xi)N(\xi)])x  -\bbe[B(\xi) W(x,\xi)q_2(\xi)] +  q_1\geq0,
\end{equation} 
where
 $$
 \|(A - \bbe[B(\xi) W(x,\xi)N(\xi)])^{-1}\|\leq \frac{1}{\bbe[\kappa(\xi)]}<+\infty.
 $$

\item[(iv)] Let
\bgeqn\label{eq:p2-1}
F(x)  := \min\big(x, (A - \bbe[B(\xi) W(x,\xi)N(\xi)])x  -\bbe[B(\xi) W(x,\xi)q_2(\xi)] +  q_1\big).
\edeqn
Then
$F$ is Lipschitz continuous and
every matrix $V_x$ in
the Clarke generalized Jacobian
$\partial F(x)$
(see definition in \cite[Section 2.6]{Cla83})
is nonsingular with $\|V_x^{-1}\|\leq \bar{d}$ for some constant $\bar{d}>0$ which is independent of $x$. 

\end{itemize}
\end{prop}

\begin{proof} We only prove Part (i) and Part (iv)
as the other
parts  follow straightforwardly from \cite{XC13}, Lemma \ref{l:matrixschur} and the implicit function theorem
\cite[Lemma 2.2]{Xu06}.

Part (i).
For this purpose,  we
 prove monotonicity of the
 infinite complementarity
  system  \eqref{eq:slcp-two1}.
Let $\langle\cdot,\cdot\rangle$ denote
the scalar product in the Hilbert space
of
$\R^n\times \Y$
equipped with ${\cal L}_2$-norm,
that is, for $x,z\in \R^n$ and $y,u\in \Y$,
$$
\langle (x,y), (z,u)\rangle := x^Tz+ \int_\Xi y(\xi)^T u(\xi) P(d\xi).
$$
Under Assumption \ref{A-AD}, for any $(x, y(\cdot)), (z, u(\cdot))\in \R^n\times \Y$, we have
$$
\begin{array}{lll}
\left\langle
\begin{pmatrix}
 A (x- z) + \bbe[B(\xi)(y(\xi) - u(\xi))] \\
 M(\xi)(y(\xi) - u(\xi)) + N(\xi)(x - z)
\end{pmatrix},
\begin{pmatrix}
 x-z\\
 y(\xi) - u(\xi)
\end{pmatrix}
 \right\rangle\\
 =
 \bbe \left[\begin{pmatrix}
x-z\\
y(\xi) - u(\xi)
\end{pmatrix}^T \begin{pmatrix}
A & B(\xi)\\
N(\xi) & M(\xi)
\end{pmatrix} \begin{pmatrix}
x-z\\
y(\xi) - u(\xi)\end{pmatrix}\right]\\
\geq
\bbe[\kappa(\xi)(\|x-z\|^2 + \|y(\xi)-u(\xi)\|^2)].
 \end{array}
$$
The existence and uniqueness of $(x^*, y^*(\cdot))$ follow from the monotonicity of problem \eqref{eq:slcp-two1}, see \cite[Theorem 12.2 and Lemma 12.2]{HaKoSc}.

Part (iv).
Let us rewrite $F(x)$ in \eqref{eq:p2-1}  as
$$
F(x) =\min(x, Ax - \bbe[B(\xi)\bar{y}(x, \xi)] +q_1),
$$
where $\bar{y}(x, \xi)$ is defined in \eqref{eq:soly-o}. By \cite[Theorem 2.1]{XC13},
$$
\begin{array}{lll}
\partial_x \bar{y}(x, \xi) &=& \inmat{conv}\left\{\lim_{z\to x} \nabla_z \bar{y}(z, \xi) = -[I - D(z, \xi)(I-M(\xi))]^{-1}D(z, \xi)N(\xi): \bar{y}(z, \xi)  \inmat{ is nondegenerate}\right\}\\
&\subseteq& \inmat{conv}\{-U_J(M(\xi))N(\xi): J\in\mathscr{J} \},
\end{array}
$$
where   ``conv'' denotes  convex hull of a set, $\partial_x \bar{y}(x, \xi)$ denotes the Clarke generalized Jacobian
with respect to $x$ and   non-degeneration  of $\bar y(z, \xi)$ means that $\{i | (M(\xi)\bar y(z, \xi) + N(\xi)z +q_2(\xi))_i= \bar y_i(z, \xi) \} = \emptyset$.
 To ease the exposition,
let $\Gamma_J(\xi) =  A- B(\xi)U_J(M(\xi))N(\xi)$.
Then
$$
\partial_x \bbe[Ax - B(\xi)\bar y(x, \xi)+q_1] 
\subseteq \bbe[\inmat{conv}\{\Gamma_J(\xi): J\in\mathscr{J} \}],
$$
where the expectation is taken in the sense of Aumann \cite{aum},
and hence
\bgeqn\label{eq:partialF}
\partial F(x) \subseteq \{\bbe[I-D + D \Upsilon(\xi)]: \Upsilon(\xi) \in\inmat{conv} \{\Gamma_J(\xi)\},  J\in\mathscr{J}, D\in\D \},
\edeqn
where $\D$ is defined immediately after Assumption \ref{A-AD}.
On the other hand, under  Assumption \ref{A-AD},
$\Gamma_J(\xi)$ is positive definite for all $\xi$
and
it follows by Lemma \ref{l:matrixschur} (ii)
$$
z^T\bbe[\Gamma_J(\xi)] z = \sum_{i=1}^n z_i(\bbe[\Gamma_J(\xi)] z)_i \geq \bbe[\kappa(\xi)]\|z\|^2, \;\; \forall z\in\R^n.
$$
This implies
$$
\max_{i\in\bar{n}}z_i(\bbe[\Gamma_J(\xi)] z)_i\geq \frac{\bbe[\kappa(\xi)]}{n}\|z\|^2,  \;\; \forall z\in\R^n.
$$
For a fixed $z\in\R^n$, let
$
i_0= \arg\max_{i\in\bar{n}} z_i (\bbe[\Gamma_J(\xi)]z)_i.
$
Then  
\begin{equation}\label{eq:kappai0}
\frac{\bbe[\kappa(\xi)]}{n}\|z\|^2 \leq z_{i_0} (\bbe[\Gamma_J(\xi)]z)_{i_0} \leq |z_{i_0}| \|\bbe[\Gamma_J(\xi)]\|\|z\|
\end{equation}
from
which we deduce
\begin{equation}\label{eq:kappai1}
|z_{i_0}|\geq \frac{\bbe[\kappa(\xi)]}{n\|\bbe[\Gamma_J(\xi)]\|}\|z\|.
\end{equation}
Moreover
\bgeqn\label{eq:zi0}
 z_{i_0}((I-D)z)_{i_0} + z_{i_0}(D\bbe[\Gamma_J(\xi)]z)_{i_0} = \left\{
 \begin{array}{ll}
 |z_{i_0}|^2, & \; \inmat{if} \;\;  D_{i_0i_0} = 0,\\
 z_{i_0}(\bbe[\Gamma_J(\xi)]z)_{i_0}, & \; \inmat{if} \;\; D_{i_0i_0} = 1.
 \end{array}
 \right.
\edeqn
Let $\Theta=
\min \left\{\left(\frac{\bbe[\kappa(\xi)]}{n\|\bbe[\Gamma_J(\xi)]\|}\right)^2, \frac{\bbe[\kappa(\xi)]}{n} \right\}$.
Then we obtain by combining \eqref{eq:kappai0}-\eqref{eq:zi0} that
$$
z_{i_0}((I-D)z)_{i_0} + z_{i_0}(D\bbe[\Gamma_J(\xi)]z)_{i_0}\geq \Theta \|z\|^2.
$$
Hence for any $z\in\R^n$, we have
$$
\displaystyle\max_{i\in\bar{n}} z_i ((I-D+D\bbe[\Gamma_J(\xi)]) z)_i  =  \displaystyle \max_{i\in\bar{n}} z_i((I-D)z)_i + z_i(D\bbe[\Gamma_J(\xi)]z)_i
 \geq  \Theta \|z\|^2,
$$
which implies that $I-D+D\bbe[\Gamma_J(\xi)]$ is a P-matrix\footnote{Recall that $A \in \R^{n\times n}$ is a P-matrix if $\max_{i\in\bar{n}} z_i(Az)_i>0$ for all nonzero $z\in\R^n$.  }  for any $D\in \D$.
Let
$$
\theta(\bbe[\Gamma_J(\xi)])=\min_{\|z\|_{\infty} = 1} \left\{\max_{i\in\bar{n}}z_i(\bbe[\Gamma_J(\xi)]z)_i\right\}.
$$
It follows from (\ref{eq:kappai0}) and (\ref{eq:kappai1}) that
$\theta(\bbe[\Gamma_J(\xi)])\geq \Theta$ and
by \cite[formula (1.6)]{XC06a},
$$
\max_{D\in \D}\|(I-D+D\bbe[\Gamma_J(\xi)])^{-1}\|_\infty \leq \frac{\max\{1, \|\bbe[\Gamma_J(\xi)]\|_\infty\}}{\theta(\bbe[\Gamma_J(\xi)])}\leq \frac{1}{\Theta}\max\{1, \|\bbe[\Gamma_J(\xi)]\|_\infty\}.
$$
Through \eqref{eq:partialF}, this implies  that every matrix in $\partial_x F(x)$ is nonsingular  and the infinity norm of its inverse is bounded by $\max\{1, \|\bbe[\Gamma_J(\xi)]\|_\infty\}/\Theta$.  
\hfill$\Box$

\end{proof}

Recall that one of the main objectives of this paper is to
develop a discretization scheme for the two-stage SLCP \eqref{eq:slcp-two1}
and we will do so in the next section.
While it is not necessarily a prerequisite, we find it is much more convenient,
at least from presentational perspective, to discuss the approach
when the support set $\Xi$ is compact.
If we made a direct assumption on the compactness of $\Xi$, it
would exclude a number of practically interesting cases where the support set is unbounded.
In what follows, we try to address the dilemma  under some moderate conditions.

Let $\epsilon$ be a small positive number and $\Xi_\epsilon$ be a compact subset of $\Xi$.
Since $B(\xi)U_J(M(\xi))N(\xi)$ and $B(\xi)U_{J(\xi)}(M(\xi))q_2(\xi)$  are integrable for all $J(\xi)\subseteq \{ 1, \cdots, n\}$, then
we can find  $\Xi_\epsilon\subset \Xi$  such that
\bgeqn\label{eq:Xiep}
\left\|\int_{\Xi \backslash \Xi_\epsilon}B(\xi)U_{J(\xi)}(M(\xi))N(\xi) P(d\xi) \right\| \leq \epsilon  \; \mbox{ and } \;  \left\|\int_{\Xi \backslash \Xi_\epsilon}B(\xi)U_{J(\xi)}(M(\xi))q_2(\xi) P(d\xi)\right\| \leq \epsilon.
\edeqn
We consider complementarity problem
\begin{equation}
\label{eq:slcp2_1ep}
\left\{
\begin{array}{ll}
0\leq x \perp A x + \bbe_{\Xi_\epsilon}[B(\xi)y(\xi)] +  q_1\geq0, \\
0\leq y(\xi) \perp M(\xi)y(\xi) + N(\xi)x  + q_2(\xi)\geq0, \;\;  \mbox{ for a.e.}\, \xi\in\Xi_\epsilon,
\end{array}\right.
\end{equation}
where  $\bbe_{\Xi_\epsilon}[H(\xi)] := \int_{\xi\in \Xi_\epsilon}H(\xi)d\xi$. Under Assumption \ref{A-AD},
it follows by Proposition \ref{p:solutiony-x}
that
 \eqref{eq:slcp2_1ep} has a unique solution (see \cite{CoPaSt92, FaPa03}).
 Let us denote the solution by
 $(x_\epsilon,y_\epsilon(\cdot))$ and substitute
 $y_\epsilon(\xi)$ obtained from the second equation of \eqref{eq:slcp2_1ep} into the first equation,
 we obtain
\begin{equation}\label{eq:slcp2_1ep-1}
0\leq x \; \perp \; (A - \bbe_{\Xi_\epsilon}[B(\xi) W(x,\xi)N(\xi)])x  -\bbe_{\Xi_\epsilon}[B(\xi) W(x,\xi)q_2(\xi)] +  q_1\geq0,
\end{equation}
where $W(x,\xi)$ is defined as in (\ref{eq:W}).

The following proposition quantifies the difference between $x_\epsilon$ and the true solution of the two-stage SLCP \eqref{eq:slcp-two1}.

\begin{prop}
Let Assumption \ref{A-AD} hold. Then there exists a positive constant $\epsilon_0$ such that
the two-stage SLCP \eqref{eq:slcp2_1ep} has a unique solution $x_\epsilon$
for all $\epsilon\in(0, \epsilon_0]$. Moreover,
there exists a positive constant
$C$ such that
\bgeqn\label{eq:p2c}
\|x^*-x_\epsilon\| \leq
C\epsilon, \,\,\,  \forall  \epsilon\in(0, \epsilon_0].
\edeqn
\end{prop}

\begin{proof}
Note that
conditions \eqref{eq:slcp2_1}   and \eqref{eq:slcp2_1ep-1} can  be rewritten respectively as \eqref{eq:p2-1} and
\bgeqn\label{eq:p2-2}
F_\epsilon(x) := \min\big(x, (A - \bbe_{\Xi_\epsilon}[B(\xi) W(x,\xi)N(\xi)])x  -\bbe_{\Xi_\epsilon}[B(\xi) W(x,\xi)q_2(\xi)] +  q_1\big)=0.
\edeqn
By \eqref{eq:Xiep}, we have
\bgeq
\|F(x)-F_\epsilon(x)\| &\leq & \|( \bbe[B(\xi) W(x,\xi)N(\xi)]-\bbe_{\Xi_\epsilon}[B(\xi) W(x,\xi)N(\xi)])x\|  \\
&&+\|\bbe[B(\xi) W(x,\xi)q_2(\xi)]   -\bbe_{\Xi_\epsilon}[B(\xi) W(x,\xi)q_2(\xi)]\| \\
&\leq& \epsilon(1+\|x\|).
\edeq
Since the solutions of \eqref{eq:p2-1} - \eqref{eq:p2-2} lie in the ball $B(0,\rho)$ of $\R^n$,  we have
$$
\sup_{\|x\|\leq \rho} \|F(x)-F_\epsilon(x)\| \leq  \epsilon(1+\rho).
$$
By Lemma \ref{l:matrixschur} (ii), $A - \bbe[B(\xi) W(x,\xi)N(\xi)]$ is positive definite and
$$
\|(A - \bbe[B(\xi) W(x,\xi)N(\xi)])^{-1}\|\leq \frac{1}{\kappa(\xi)}.
$$
Thus, there exists a positive constant  $\epsilon_0$ such that for all
$
\epsilon\in (0, \epsilon_0],
$
we have
  $$
  A - \bbe_{\Xi_{\epsilon}}[B(\xi) W(x,\xi)N(\xi)]\succ 0
  \;\;\inmat{ and } \;\;
\|(A - \bbe_{\Xi_\epsilon}[B(\xi) W(x,\xi)N(\xi)])^{-1}\|\leq \frac{2}{\kappa(\xi)}.
$$
 Since $\|\bbe_{\Xi_\epsilon}[B(\xi) W(x,\xi)q_2(\xi)]\|$ is also bounded,
there is a positive number $\alpha_1$ such that
$\|x_\epsilon\|\leq \alpha_1$
for all $\epsilon\in(0, \epsilon_0]$.
This enables us to bound  $\rho$ by a positive constant independent of
$x_\epsilon$.

By Proposition \ref{p:solutiony-x} (iv),
there exists a positive constant $\alpha_2$ such that
$$
\|x_\epsilon-x^*\| \leq \alpha_2 \|F(x_\epsilon)-F(x^*)\| = \alpha_2  \|F(x_\epsilon)-F_\epsilon(x_\epsilon)\| \leq
\alpha_2 \epsilon(1+\alpha_1).
$$
We obtain the conclusion by setting
$C=\alpha_2(1+\alpha_1)$.
\hfill $\Box$

The proposition says that the solution obtained from solving (\ref{eq:slcp2_1ep}) is close to the solution of (\ref{eq:slcp-two1})
when $\epsilon$ is set sufficiently small. This means
that we
can
trim off the tail of
the
probability distribution $P$.

\end{proof}

\section{A discretization scheme}

In this section, we
move on to discuss
discretization approaches for the two-stage
SLCP \eqref{eq:slcp-two1}.
The key challenge
is that the second stage of SLCP  \eqref{eq:slcp-two1}
comprises
an infinite number of  complementarity problems  when $\xi$ is continuously distributed.
Our idea here is to divide
the support set $\Xi$ of $\xi$ into small subsets $\{\Xi_i\}$ and
set ${\bf y}_i\equiv y(\xi)$ over each of the subset $\Xi_i$. This will effectively
reduce the infinite number of complementarity problems at the second stage to a finite number.
%discretize
We then attach a probability to each of the subset and consequently
discretize the probability distribution $P$
and the two-stage SLCP \eqref{eq:slcp-two1}.
Throughout this section, we assume that $\Xi$ is a compact and convex set.

\subsection{Description of the discretization scheme}

 Let $\{\Xi^K_i\}$  be a partition of the support set  $\Xi$, that is, $\Xi^K_i$
 is a compact and convex subset of $\Xi$ such that,
$$
\bigcup_{i=1}^K \Xi^K_i = \Xi, \;\; \inmat{int} \Xi^K_i \cap \inmat{int} \Xi^K_j = \emptyset,
 \;\;\forall \;  i \neq j, \;\; i, j\in\bar{K}, 
$$
where $\inmat{int} S$ denotes the interior of  $S$.
Note that since $\Xi$ is assumed to be a compact set, each $\Xi^K_i$ is also a compact set.
Let
\bgeqn\label{eq:Xi12}
\bbe_{\Xi^K_i}[H(\xi)] = \frac{1}{p_i^K}\int_{\xi\in \Xi^K_i}H(\xi) P(d\xi) \;\; \inmat{with} \;\; p_i^K = P(\Xi^K_i)
\edeqn
for $H(\xi) = M(\xi), N(\xi), B(\xi)$ and $q_2(\xi)$. Let
\bgeqn\label{eq:Delta}
\Delta(\Xi_i^K)= \displaystyle{\max_{\xi_1, \xi_2\in \Xi_i^K}}\|\xi_1- \xi_2\|
\edeqn
denote
 the diameter of $\Xi_i^K$.
  We require $\max_{i\in\bar{K}} \Delta(\Xi_i^K)  \to 0$ as $K\to\infty$.

Let $x\in \R^n$ be fixed.  For $i\in\bar{K}$,
we consider the linear complementarity problem
 \bgeqn\label{eq:cpki}
 0\leq {\bf y}_i \perp \bbe_{\Xi^K_i}[M(\xi)]{\bf y}_i + \bbe_{\Xi^K_i}[N(\xi)]x  + \bbe_{\Xi^K_i}[q_2(\xi)]\geq0.
 \edeqn
 Under Assumption \ref{A-AD},
 $M(\xi)\succ 0$ for all $\xi\in \Xi$ and hence $\bbe_{\Xi^K_i}[M(\xi)] \succ 0$.
 This ensures problem \eqref{eq:cpki} has a unique solution.
 By \cite{XC13} and Lemma \ref{l:matrixschur}, we can write the solution of \eqref{eq:cpki} as
 \begin{equation}\label{eq:yik}
 \bar{{\bf y}}_i^K(x) := - W^K_i(x)(\bbe_{\Xi^K_i}[N(\xi)]x + \bbe_{\Xi^K_i}[q_2(\xi)]),
\end{equation}
where
$$
W^K_i(x) : = [I - D^K_i(x) (I- \bbe_{\Xi^K_i}[M(\xi)])]^{-1}D^K_i(x),
$$
$D^K_i(x)\in \D$ is an $m\times m$-dimension
diagonal matrix with
$$
(D^K_i(x))_{jj} : = \left\{
\begin{array}{ll}
1, & \mbox{ if } (\bbe_{\Xi^K_i}[N(\xi)]x + \bbe_{\Xi^K_i}[M(\xi)]\bar{{\bf y}}^K_i(x) + \bbe_{\Xi^K_i}[q_2(\xi)])_j \leq (\bar{{\bf y}}^K_i(x))_j,\\
0, & \mbox{ otherwise}.
\end{array}
 \right.
$$
Let $\mathbf{1}_A(a)$ denote the indicator function with $\mathbf{1}_A(a)=1$ for $a\in A$ and $0$ otherwise. Let
\bgeqn\label{eq:yk}
\bar y^K(x, \xi) = \sum_{i=1}^K  \bar{{\bf y}}_i^K(x)\mathbf{1}_{\Xi^K_i}(\xi),
\edeqn
that is,
$\bar y^K(x, \xi) = \bar{{\bf y}}_i^K(x), \;\; \forall \xi\in \Xi_i^K.$

We use $\bar y^K(x, \xi)$ as a piecewise
step-like approximation to the true solution $\bar y(x,\xi)$
of
the second stage of SLCP
\eqref{eq:slcp-two1}. Substituting $\bar y^K(x, \xi)$ into the first stage of SLCP  \eqref{eq:slcp-two1},
we obtain
\begin{eqnarray}\label{eq:slcp-two1N-1}
0\leq x \perp A x + \bbe[B(\xi)\bar{y}^K(x, \xi)] +  q_1\geq0.
\end{eqnarray}

\begin{rem}
 Problems \eqref{eq:cpki}$-$\eqref{eq:slcp-two1N-1}   provide a discrete approximation of two-stage SLCP
 \eqref{eq:slcp-two1}. To see this, let $p_i^K:=P(\Xi_i^K)$. Then we can write  \eqref{eq:cpki}$-$\eqref{eq:slcp-two1N-1}  as
\begin{equation}
\label{eq:slcp-two1N}
\left\{
\begin{array}{ll}
0\leq x \perp A x + \sum_{i=1}^Kp_i^K\bbe_{\Xi_i^K}[B(\xi)]{\bf y}_i +  q_1\geq0,\\
0\leq {\bf y}_i \perp \bbe_{\Xi^K_i}[M(\xi)]{\bf y}_i + \bbe_{\Xi^K_i}[N(\xi)]x  + \bbe_{\Xi^K_i}[q_2(\xi)]\geq0,  \; \;\; i
\in \bar{K}.
\end{array}\right.
\end{equation}
 This is a two-stage SLCP with discrete distribution.
 The discretization scheme should be distinguished from the well-known sample average
 approximation scheme where the second stage solution
 is restricted only to the sample points each of which is attached with equal probability.
 Our approach is more accurate by exploiting the information of problem data in each
 set $\Xi_i^K$ albeit at the cost of calculating $p_i^K$.

\end{rem}

\subsection{Qualitative and quantitative convergence analysis}

Let $(x^K, {\bf y}^K)$ denote the solution of
(\ref{eq:slcp-two1N}) whereby we write
${\bf y}^K$ for $({\bf y}_1^K,\cdots,{\bf y}_K^K)$.
Let
\bgeqn
\label{eq:yK-rem-3.1}
 y^K(\xi) = \sum_{i=1}^K  {\bf y}_i^K\mathbf{1}_{\Xi^K_i}(\xi).
\edeqn
We investigate
convergence of $(x^K, y^K(\cdot))$ to $(x^*,y^*(\cdot))$,
the true solution of the two-stage SLCP \eqref{eq:slcp-two1}
  as $\max_{k\in\bar{K}} \Delta(\Xi_i^K)$ goes to zero.
  At this point, it might be helpful to emphasize the difference between
  $y^K(\xi)$ and $\bar{y}^K(x,\xi)$ defined in (\ref{eq:yk}):  the latter depends on each fixed $x$ whereas
the former does not depend on $x$.
Using \eqref{eq:yik} and \eqref{eq:yk}, we have
\bgeqn
\bar{y}^K(x, \xi) = -\sum_{i=1}^K W_i^K(x)(\bbe_{\Xi^K_i}[N(\xi)]x + \bbe_{\Xi^K_i}[q_2(\xi)])  \mathbf{1}_{\Xi^K_i}(\xi).
\edeqn
Substituting the explicit form of $\bar{y}^K(x,\xi)$ above into the first
equation of (\ref{eq:slcp-two1}),
we obtain
\begin{equation}\label{eq:slcp2_1N}
0\leq x \; \perp \; \left(A - \bbe\left[B(\xi) \left(\sum_{i=1}^K\mathbf{1}_{\Xi_i^K}(\xi) W^K_i(x)\bbe_{\Xi_i^K}[N(\xi)]\right)\right]\right)x  - Q^K(x) \geq 0,
\end{equation}
where
$$
Q^K(x) := \bbe\left[B(\xi) \left(\sum_{i=1}^K\mathbf{1}_{\Xi_i^K}(\xi) W^K_i(x)\bbe_{\Xi_i^K}[q_2(\xi)]\right)\right] -  q_1.
$$

The following theorem states
the convergence of $(x^K, y^K( \xi))$
to $(x^*, y^*(\xi))$ as $K\to \infty$.

\begin{thm}\label{t:convergence}
Under
Assumption \ref{A-AD},
 the following assertions hold.
\begin{itemize}
\item [(i)]
The complementarity problem \eqref{eq:slcp-two1N} has a unique solution $(x^K, {\bf y}^K)$.

\item [(ii)] If, in addition,
 $\max_{i\in\bar{K}} \Delta(\Xi_i^K)  \to 0$, then
   $\{(x^K, y^K( \cdot))\}$ is bounded on $\R^n\times \Y$, 
where the boundedness of $y^K( \cdot)$ is in the sense of the norm topology of $\mathcal{L}_1(\Y)$.

\item [(iii)] $\{x^K, y^K( \cdot)\}$ converges to the true solution $(x^*, y^*( \cdot))$ of
problem \eqref{eq:slcp-two1}, where the convergence of $\{y^K( \cdot)\} \to y^*( \cdot) $ is in the sense of the norm topology of $\mathcal{L}_2(\Y)$.
\end{itemize}
\end{thm}

\begin{proof}
Part (i).
 By Assumption \ref{A-AD},
$$
(z^T, u^T)
\begin{pmatrix}
A &\bbe_{\Xi_i^K}[B(\xi)]\\
\bbe_{\Xi_i^K}[N(\xi)]& \bbe_{\Xi_i^K}[M(\xi)]
\end{pmatrix}
\begin{pmatrix}
z\\
u
\end{pmatrix}\geq \bbe_{\Xi_i^K}[\kappa(\xi)](\|z\|^2 + \|u\|^2).
$$
Analogous to Proposition \ref{p:solutiony-x},
we can demonstrate that
the discretized  two-stage SLCP \eqref{eq:slcp-two1N} has a
 unique solution.

Part (ii). By definition,
$x^K$ is a solution to the first stage of SLCP  \eqref{eq:slcp-two1N} and it is the unique solution for all $K$.
Moreover
$$
W_i^K(x^K) = (I - D^K_i(x^K) (I- \bbe_{\Xi^K_i}[M(\xi)]))^{-1}D^K_i(x^K),
$$
where $D_i^K(x^K)\in \D$, $i\in \bar{K}$.
For any given $
\bar{D}_i
\in \D$,  let
$$
\bar{W}_i^K =  (I - \bar{D}_i (I- \bbe_{\Xi^K_i}[M(\xi)]))^{-1}\bar{D}_i \;\; \inmat{and} \;\;
\tilde{W}_i^K(\xi)=(I - \bar{D}_i (I- M(\xi)))^{-1}\bar{D}_i.
$$
Under Assumption \ref{A-AD},
 $\|\tilde{W}_i^K(\xi)\| =\|U_J(M(\xi))\|\leq \frac{1}{\kappa(\xi)}$ for some subset $J\in \mathscr{J}$.
 Thus $\|\bar{W}_i^K\|\leq \frac{1}{\kappa(\xi)}+1$ for all $K$.
Let
\begin{equation}
R^K
=\bbe \left[ B(\xi) \left(\sum_{i=1}^K\mathbf{1}_{\Xi_i^K}(\xi) \left(\bar{W}^K_i\bbe_{\Xi_i^K}[N(\xi)] -  \tilde{W}_i^K(\xi)N(\xi)\right) \right) \right].
\end{equation}
For $\bar{D}_i = D_i^K(x^K)$,
$x^K$ satisfies
\begin{equation}\label{eq:slcp2_1Nbar}
0\leq x \; \perp \; \left(A - \bbe\left[B(\xi) \left(\sum_{i=1}^K\mathbf{1}_{\Xi_i^K}(\xi) \bar{W}^K_iN(\xi)\right)\right]\right)x  - \bar{Q}^K \geq 0,
\end{equation}
where
$$
\bar{Q}^K = \bbe\left[B(\xi) \left(\sum_{i=1}^K\mathbf{1}_{\Xi_i^K}(\xi) \bar{W}^K_i\bbe_{\Xi_i^K}[q_2(\xi)]\right)\right] -  q_1.
$$
Moreover, \eqref{eq:slcp2_1Nbar} can be written as
\begin{equation}\label{eq:slcp2_1Nt}
0\leq x \; \perp \; \left(A - \bbe\left[B(\xi) \left(\sum_{i=1}^K\mathbf{1}_{\Xi_i^K}(\xi) \tilde{W}^K_i(\xi)N(\xi)\right)\right]+R^K\right)x  - \bar{Q}^K \geq 0.
\end{equation}
Note that for any $\xi\in \Xi_i^K$,
\begin{equation}\label{eq:WNdifference}
\bar{W}^K_i\bbe_{\Xi_i^K}[N(\xi)] -  \tilde{W}_i^K(\xi)N(\xi) = \bar{W}^K_i(\bbe_{\Xi_i^K}[N(\xi)] - N(\xi)) + N(\xi) (\bar{W}^K_i - \tilde{W}_i^K).
\end{equation}
Since $\max_{i\in\bar{K}} \Delta(\Xi_i^K)  \to 0$, and  both $M(\cdot)$ and $N(\cdot)$ are continuous over $\Xi$, we have
$$
\sup_{\xi\in \Xi_i^K}\|M(\xi) - \bbe_{\Xi_i^K}[M(\xi)]\| \to 0 \;\;
\mbox{ and } \;\;
\sup_{\xi\in \Xi_i^K}\|N(\xi) - \bbe_{\Xi_i^K}[N(\xi)]\| \to 0.
$$
Moreover, under Assumption \ref{A-AD},
it follows by Lemma \ref{l:matrixschur},
$\|\bar{W}_i^K\|\leq \frac{1}{\min_{\xi\in\Xi}\kappa(\xi)}$ (note that $\kappa(\xi)$ is positive continuous  over the compact set $\Xi$,
$\min_{\xi\in \Xi} \kappa(\xi)>0$). Thus,
 the first term
 at the right hand side of \eqref{eq:WNdifference}  goes to zero as $K\to\infty$.
  Likewise, we can show that the second term at the right hand side of \eqref{eq:WNdifference}  goes to zero as $K\to\infty$.
Summarizing the discussions above, we
are able to claim,
  by  the Lebesgue Dominated Convergence Theorem that,
$$
\lim_{K\to\infty}\bbe\left[B(\xi) \left(\sum_{i=1}^K\mathbf{1}_{\Xi_i^K}(\xi) \bar{W}^K_i(\xi)N(\xi)\right)\right]  =\bbe\left[B(\xi) \left(\sum_{i=1}^K\mathbf{1}_{\Xi_i^K}(\xi) \tilde{W}^K_i(\xi)N(\xi)\right)\right]
$$
and
$R^K\to 0$ as $K\to \infty$.  By Lemma \ref{l:matrixschur},
$$
x^T\left(A - \bbe\left[B(\xi) \left(\sum_{i=1}^K\mathbf{1}_{\Xi_i^K}(\xi) \bar{W}^K_i\bbe_{\Xi_i^K}[N(\xi)]\right)\right]\right)x\geq \frac{1}{2}\bbe[\kappa(\xi)]\|x\|^2
$$
and
$$
\left\|\left(A - \bbe\left[B(\xi) \left(\sum_{i=1}^K\mathbf{1}_{\Xi_i^K}(\xi) \bar{W}^K_i\bbe_{\Xi_i^K}[N(\xi)]\right)\right]\right)^{-1}\right\|\leq \frac{2}{\bbe[\kappa(\xi)]}
$$
for any $\{\bar{D}_i\}\subset \D$. This entails
\bgeqn
\left\|\left(A - \bbe\left[B(\xi) \left(\sum_{i=1}^K\mathbf{1}_{\Xi_i^K}(\xi) W^K_i(x)\bbe_{\Xi_i^K}[N(\xi)]\right)\right]\right)^{-1}\right\|\leq \frac{2}{\bbe[\kappa(\xi)]}
\label{eq:ABWK}
\edeqn
for all $x\in \R^n_+$.
On the other hand,
the boundedness of
$\bar{Q}^K$ implies that $\|Q^K(x)\|$ is
uniformly bounded
w.r.t. $x$. Together with (\ref{eq:ABWK}),  we are able to claim the boundedness of $\{x^K\}$.
This together with integrable boundedness of $M(\xi)^{-1}$, $N(\xi)$ and $q_2(\xi)$,
ensures that ${\bf y}^K$ is bounded and so is
$y^K(\cdot)$.

Part (iii). Let $\hat{x}$ be a cluster point of the sequence
$\{x^K\}$
and assume without loss of generality that $x^K\to \hat{x}$.
For any fixed $\xi\in \Xi$, there exists $\{i_K\}$ such that $\xi\in\displaystyle{\bigcap_K}\Xi_{i_K}^K$. By the implicit function theorem \cite[Lemma 2.2]{Xu06},
$y^K(\xi)$ converges to $\hat{y}( \xi)$ which satisfies
 \bgeqn\label{eq:cpkhat}
0\leq \hat{y}(\xi) \perp M(\xi)\hat{y}( \xi) + N(\xi)\hat{x}  + q_2(\xi)\geq0, \;\; \inmat{for}\; \; {\rm a.e.} \;\;  \xi\in \Xi.
 \edeqn
Moreover,  since 
 $y^K(\cdot)$ is bounded,  by the Lebesgue Dominated Convergence Theorem,
$$
\lim_{K\to\infty} \bbe\left[B(\xi)
y^K(\xi)
\right] = \bbe\left[\lim_{K\to\infty}B(\xi)
y^K(\xi)
\right] = \bbe[B(\xi)\hat{y}( \xi)].
$$
 Through (\ref{eq:yk}) and (\ref{eq:slcp-two1N-1}), we have
 \begin{eqnarray}\label{eq:slcp-two1N-1a}
0\leq \hat{x} \perp A \hat{x} + \bbe[B(\xi)\hat{y}( \xi)] +  q_1\geq0.
\end{eqnarray}
The  equation above coincides with the first stage of SLCP (\ref{eq:slcp-two1}). Since
the second stage problem
has a unique solution, $(\hat{x},\hat{y}(\xi))$
coincides with $(x^*,y^*(\xi))$ a.e.
The proof is complete. \hfill $\Box$
 \end{proof}

In what follows, we
take a step further to
quantify the discrepancy
between $(x^K, y^K( \xi))$
and $(x^*, y^*(\xi))$ as $K\to\infty$.
For this purpose, we require the underlying coefficient matrices and vectors to be Lipschitz continuous
w.r.t. $\xi$.

\begin{assu} \label{A:MNXI}
$M(\cdot)$, $N(\cdot)$, $q_2(\cdot)$ and $B(\cdot)$ are Lipschitz continuous over a compact set containing  $\Xi$ with  Lipschitz constant   $L$.
\end{assu}

Under  Assumptions \ref{A:MNXI}, it is easy to show that
\bgeqn\label{eq:Hgoto0}
\|H(\xi) - \bbe_{\Xi^K_i}[H(\xi)]\| \leq L\Delta(\Xi_i^K),
\forall \xi\in \Xi_i^K,
\edeqn
for $H(\xi)= M(\xi), N(\xi), B(\xi)$ and $ q_2(\xi)$.

\begin{thm}
\label{t-quantitative-sect3}
Under Assumptions \ref{A-AD} and \ref{A:MNXI}, there exist
  a positive number $\gamma\geq0$ and
   nonnegative integrably bounded functions $c(\xi)$ and $h(\xi)$ such that
\bgeqn
\|x^K - x^*\|
\leq \gamma \bbe[\|B(\xi)\|c(\xi)]L\max_{i\in\bar{K}} \Delta(\Xi_i^K)
\label{eq:xk-x*-quantitative}
\edeqn
and
\bgeqn
\|y^K( \xi) - y^*( \xi)\|
\leq h(\xi)L\max_{i\in\bar{K}}\Delta(\Xi_i^K), \;\; \inmat{for} \;\;\forall \xi\in\Xi. 
\label{eq:yk-quantitative}
\edeqn
\end{thm}

   \noindent
\textbf{Proof.} We first prove \eqref{eq:xk-x*-quantitative} and proceed it  in two steps.

\noindent
Step 1.  By Theorem \ref{t:convergence},  there exists a compact set $X\subset \R^n$
which encompasses  $\{x^K\}$ and $x^*$. For any fixed $\xi\in\Xi$ and $x\in X$, we consider the second stage complementarity problem
\bgeqn
0\leq y \perp M(\xi)y + N(\xi)x  + q_2(\xi)\geq0.
\label{eq:LCP-y-xi}
\edeqn
As we discussed in the previous section, (\ref{eq:LCP-y-xi}) has a unique solution
$y$ under Assumption \ref{A-AD}. Moreover, we can write (\ref{eq:LCP-y-xi}) equivalently as
$$
\Phi(y,x,\xi) :=\min(y,M(\xi)y + N(\xi)x  + q_2(\xi)) =0.
$$
Under Assumption \ref{A:MNXI}, $\Phi$ is locally Lipschitz continuous
near $(y,x,\xi)$.  By the implicit function theorem
\cite[Lemma 2.2]{Xu06},
\eqref{eq:LCP-y-xi}  has a unique
locally Lipschitz continuous function $\bar{y}(\tilde x,  \tilde \xi)$ such that
$\bar{y}(x,\xi) = y$ and $
 \Phi(\bar{y}(\tilde x, \tilde \xi), \tilde x, \tilde \xi)=0
$
for all
$(\tilde x, \tilde \xi)$ close to $(x,\xi)$.
  Likewise, we can show that there is a unique
locally Lipschitz continuous
$\bar{{\bf y}}^K_{i} (x)$
which satisfies
\bgeqn
0\leq {\bf y}_i \perp \bbe_{\Xi^K_i}[M(\xi)]\mathbf{1}_{\Xi_i^K}(\xi) {\bf y}_i + \bbe_{\Xi^K_i}[N(\xi)]\mathbf{1}_{\Xi_i^K}(\xi) x
+ \bbe_{\Xi^K_i}[q_2(\xi)]\mathbf{1}_{\Xi_i^K}(\xi) \geq0,
\label{eq:yK-Xi-i}
\edeqn
for $\xi\in \Xi_i^K, i\in \bar{K}$.
Note that the solution to (\ref{eq:yK-Xi-i}) may be represented as in (\ref{eq:yik}).

 With \eqref{eq:Hgoto0},
we may regard (\ref{eq:yK-Xi-i})  as a perturbation of \eqref{eq:LCP-y-xi}.
Let $\eta\in (0,1)$ be a constant and
$K$  sufficiently large such that
\bgeq
\bbe_{\Xi_i^K}[M(\xi)]\in
\{Q | \beta(\xi) \|M(\xi)-Q\|\leq \eta\}, \;\;
\inmat{for}\; i\in\bar{K},
\edeq
where
$$
\beta(\xi) :
=\max_{D\in {\cal D}} \|(I- D+DM(\xi))^{-1}D\|.
$$
Let $\alpha(\xi)  = \beta(\xi)/(1-\eta)$.
By \cite[Theorem 2.8]{XC06}, for $\xi \in \Xi^K_i$,
\bgeqn
 \|\bar{{\bf y}}_i^K (x) -  \bar{y}(x, \xi)\| &\leq & \alpha^2(\xi) \|(-N(\xi)x - q_2(\xi))_+\| \|\bbe_{\Xi_i^K}[M(\xi)] - M(\xi)\|
 \nonumber\\
 && + \alpha(\xi) \|\bbe_{\Xi_i^K}[N(\xi)]x - N(\xi)x + \bbe_{\Xi_i^K}[q_2(\xi)] - q_2(\xi)\|
 \nonumber\\
 & \leq & L(\alpha^2(\xi) (\|N(\xi)\|\|x\| + \|q_2(\xi)\|) + \alpha(\xi)(\|x\|+1))\Delta(\Xi_i^K)\nonumber\\
 &=& Lh_1(x, \xi) \Delta(\Xi_i^K),
 \label{eq:yyk}
\edeqn
 where
 $
 h_1(x, \xi) :=\alpha^2(\xi) (\|N(\xi)\|\|x\| + \|q_2(\xi)\|) + \alpha(\xi)(\|x\|+1).
 $

 Under Assumption \ref{A-AD}, $\bbe[h_1(x, \xi)]<\infty$.
 The error bound holds for all $i\in\bar{K}$.

 \noindent
 Step 2.  By substituting $\bar{y}(x,\xi)$ and $\bar{{\bf y}}_i^K(x)$
 into the first equation of
 \eqref{eq:slcp-two1} and \eqref{eq:slcp-two1N} respectively, we have
  $$
 0\leq x \perp A x + \sum_{i=1}^Kp_i^K\bbe_{\Xi_i^K}[B(\xi)\bar{y}(x, \xi)]+  q_1\geq0
 \,\, {\rm  and} \,\,
 0\leq x \perp A x + \sum_{i=1}^Kp_i^K\bbe_{\Xi_i^K}[B(\xi)]\bar{{\bf y}}_i^K(x) +  q_1\geq0.
$$
We write them equivalently as
$$
F(x) := \min\left\{ x, Ax+ \sum_{i=1}^Kp_i^K\bbe_{\Xi_i^K}[B(\xi)\bar{y}(x, \xi)]+  q_1 \right\} = 0
$$
and
$$
F^K(x) : = \min\left\{ x, Ax+ \sum_{i=1}^Kp_i^K\bbe_{\Xi_i^K}[B(\xi)]\bar{{\bf y}}_i^K(x) +  q_1\right\} = 0.
$$
Let $X$ be defined as in Step 1.
Since $\Xi$ is compact and $\bar{y}(x,\xi)$ is continuous over $X\times \Xi$, there is
a positive constant $\sigma$ such that
$\|\bar{y}(x,\xi)\|\leq \sigma$ for all $(x,\xi)\in X\times \Xi$.
Using \eqref{eq:Hgoto0}, we have
\bgeq
\|F(x) - F^K(x)\|
& \leq &
 \left\|\sum_{i=1}^Kp_i^K\bbe_{\Xi_i^K}[B(\xi)\bar{y}(x, \xi)] - \sum_{i=1}^Kp_i^K\bbe_{\Xi_i^K}[B(\xi)]\bar{{\bf y}}_i^K(x)\right\| \\
& \leq &
  \sum_{i=1}^Kp_i^K\bbe_{\Xi_i^K}[\|B(\xi)\|\|\bar{y}(x,\xi)-\bar{{\bf y}}_i^K(x)\|]  \\
  &\leq &  L\max_{i\in\bar{K}}\Delta(\Xi_i^K)\left(
  \sum_{i=1}^Kp_i^K\bbe_{\Xi_i^K}[\|B(\xi)\|h_1(x, \xi)]\right)  \;\;\;  \inmat{(by \ref{eq:yyk})}\\
&\le& L\bbe[
\|B(\xi)\|c(\xi)]\max_{i\in\bar{K}}\Delta(\Xi_i^K),
\edeq
where
$$
c(\xi) :=
 \alpha^2(\xi)(\|N(\xi)\|\|X\| + \|q_2(\xi)\|) + \alpha(\xi)(\|X\|+1)
$$
and
$\|X\|:=\max\{\|x\|: x\in X\}$.
By Proposition \ref{p:solutiony-x} (iv)
and \cite[Lemma 2.2]{Xu10},
\bgeq
\|x^K - x^*\| \leq \gamma
\sup_{x\in X}\|F(x) - F^K(x)\| \leq L\gamma \bbe[\|B(\xi)\|c(\xi)]\max_{i\in\bar{K}}\Delta(\Xi_i^K)
\edeq
for any positive number $\gamma \geq \frac{1}{\bbe[\kappa(\xi)]}$.
 This completes the proof of (\ref{eq:xk-x*-quantitative}).

Next we prove \eqref{eq:yk-quantitative}. Using the established error bound
   (\ref{eq:xk-x*-quantitative}) for the $x$-component of the SLCP solutions, we can again use
 \cite[Theorem 2.8]{XC06} for the $y$-component of the solutions,
 \begin{equation}
 \begin{array}{lll}
\|y^K (\xi) -  y^*(\xi)\| &= & \|\bar{y}^K (x^K, \xi) -  \bar{y}(x^*, \xi)\| \\
&\leq & \alpha^2(\xi) \|(-N(\xi)x^* - q_2(\xi))_+\| \|\bbe_{\Xi_i^K}[M(\xi)] - M(\xi)\| \\
 && + \alpha(\xi) \|\bbe_{\Xi_i^K}[N(\xi)]x^K - N(\xi)x^* + \bbe_{\Xi_i^K}[q_2(\xi)] - q_2(\xi)\|  \\
 & \leq & Lh(\xi)\displaystyle{\max_{i\in\{1, \cdots, K\}}}\Delta(\Xi_i^K),
 \end{array}
 \end{equation}
 where  $
 h(\xi) := \alpha^2(\xi) (\|N(\xi)\|C_0 + \|q_2(\xi)\|) + \alpha(\xi)C_0
 $  with  $C_0 = (\|X\|+1) + \gamma \bbe[\|B(\xi)\|\|c(\xi)\|]$. It is easy to see that $\bbe[h(\xi)]<+\infty$ under Assumption \ref{A-AD}.
\hfill $\Box$

It might be helpful to discuss how the partition of
 $\Xi$ is made.
If $\xi$ is a single random variable, then we may divide the interval of
the support set $\Xi$ evenly into $K$ subintervals. However,
if $\xi$ is a random vector which has several components, then
$K$ might have to be very large in order to reduce the size of $\Xi_i^K$.
In that case, it would be sensible to use Monte Carlo sampling to generate
a set of points $\Xi^K: =\{\xi^1,\cdots,\xi^K\}$ and use them to develop the
 Voronoi partition  of $\Xi$, that is,
	\begin{equation}\label{eq:XiiK}
		\Xi_i^K\subseteq\left\{\xi\in\Xi:\|\xi-\xi^i\|=\min_{k\in\bar{K}} \|\xi-\xi^k\|
		\right\}\quad \text{ for }  \  i\in\bar{K},
	\end{equation}
	are pairwise disjoint subsets forming a partition of $\Xi$.

\subsection{Progressive hedging method  (PHM)}

The discretized two-stage SLCP \eqref{eq:slcp-two1N} is a deterministic LCP which may be solved by any existing solvers.
However, when $K$ is large, it might be more efficient to solve \eqref{eq:slcp-two1N} with the well known  PHM \cite{rw1991} which  exploits the two-stage structure. Note that  
PHM  is an iterative approach which solves finite scenario multi-stage stochastic programming problems at each scenario and then
average  them to get a feasible solution at each iterate.  The main advantage of the approach is that the scenario based solutions can be obtained in parallel computation.
 Recently, Rockafellar and Sun \cite{rs2017} extend the method to finite scenario multi-stage stochastic variational inequalities.
Here we describe how to apply PHM to solve \eqref{eq:slcp-two1N}.

Let $\Xi^K=\{\Xi_1^K,\cdots,\Xi_K^K\}$ and
 $\Omega_i^K= \xi^{-1}(\Xi_i^K)$ for $i\in\bar{K}$.
Let $\tilde{\Omega}= \{\Omega_1^K,\cdots,\Omega_K^K\}$ and $\tilde{\mathscr{B}}$ the  sigma algebra over
$\tilde{\Omega} $. Let $\tilde{P}$ be a probability measure over the measurable space $\{\tilde{\Omega},\tilde{\mathscr{B}}\}$
and $U:\tilde{\Omega}\to 2^{\Xi^K}$ be a random variable (set-valued mapping indeed) with $\tilde{P}(U=\Xi_i^K):=p_i^K$, where $p_i^K$ is defined as in \eqref{eq:Xi12}.
By slightly abusing the notation, we also regard $\tilde{P}$ as a probability measure over $\{\Xi^K,\mathscr{B}^K\}$, where $\mathscr{B}^K$ is Borel sigma algebra over $\Xi^K$, with
$P(\Xi_i^K)=p_i^K$ for $i\in\bar{K}$.
 Let
 $ (x(\cdot), y(\cdot))$ be a measurable mapping from  $\{\Xi^K,\mathscr{B}^K\}$ to
   $ \R^{m+n} $,
  where $x(U) = \sum_{i=1}^Kx_i{\bf 1}_{\Xi^K_i}(U)$ and $y(U) = \sum_{i=1}^Ky_i{\bf 1}_{\Xi^K_i}(U)$. The linear space  $\mathcal{L}$ consisting of all such mappings $z(\cdot)$ from $\Xi^K$ to $\R^{n+m}$ is given with the expectational inner product.
Let $\mathscr{Z}$ denote the space of all measurable functions defined as such.
Define the bilinear product
$$
\langle (x(\cdot), y(\cdot)), (z(\cdot), u(\cdot))  \rangle
=
\bbe_{\tilde P}[(x(U), y(U))^T(z(U), u(U))] =
\sum_{i=1}^Kp_i^K (x_i^Tz_i + y_i^Tu_i),
$$
where $x_i,z_i\in\R^n$ and $y_i,u_i\in\R^m$.
Let $\N$ be the space of
all measurable functions of form $ (x, y(U))$, where $x$ is independent of $U$.
Then we may view $\N$ as a subspace of $\mathscr{Z}$ where the $x$-component is made deterministic (scenario free).
Let
$$
\M = \N^{\perp} = \{w(\cdot) = (w_1(\cdot), w_2(\cdot))\in \mathscr{Z}| \langle (x, y), (w_1, w_2)\rangle =0,\;\; \forall (x, y(\cdot)) \in \N
\}.
$$
Then $w(\cdot) \in \M$ implies that $\langle x,w_1\rangle =\bbe_{\tilde P}[x^Tw_{1}(U)]=0$ and  $\bbe_{\tilde P}[y(U)^Tw_2(U)] =0$ for all $(x, y(\cdot)) \in \N$ and therefore
$\bbe_{\tilde P}[w_{1}(U)] = 0, w_{2i}=0, i\in\bar{K}$.
In what follows, we describe  PHM for  solving \eqref{eq:slcp-two1N}  as follows.

For  $i\in\bar{K}$,
let $\tilde B(\Xi_i^K) =\bbe_{\Xi_i^K}[B(\xi)]$, $\tilde M(\Xi_i^K) = \bbe_{\Xi^K_i}[M(\xi)]$, $\tilde N(\Xi_i^K) = \bbe_{\Xi^K_i}[N(\xi)]$ and $\tilde q_2(\Xi_i^K) = \bbe_{\Xi^K_i}[q_2(\xi)]$. Then the discrete two-stage SLCP \eqref{eq:slcp-two1N} can be reformulated as:
\begin{equation}
\label{eq:slcp-two1NXi}
\left\{
\begin{array}{ll}
0\leq x \perp A x + \bbe_{\tilde P}[\tilde B(U) y(U)] +  q_1\geq0,\\
0\leq  y(U) \perp \tilde M(U) y(U) + \tilde N(U)x  + \tilde q_2(U)\geq0, \;\; U\in \Xi^K.
\end{array}\right.
\end{equation}

\begin{alg} [PHM] Given initial points  $(x^0,y^0)\in \N$ and $w^0\in\M$ with $x^0(U)=\sum_{i=1}^Nx_i^0{\bf 1}_{\Xi_i^K}(U)$, $y^0(U)=\sum_{i=1}^Ny_i^0{\bf 1}_{\Xi_i^K}(U)$ and $w^0(U)=\sum_{i=1}^Nw_i^0{\bf 1}_{\Xi_i^K}(U)$, $i\in\bar{K}$.  Let $r>0$ fixed  and $\nu=0$.

\textbf{Step 1.} For $i\in\bar{K}$, solve the LCP
\begin{equation}
\label{eq:slcp-two1N-PH}\left\{
\begin{array}{ll}
0\leq x_i \perp A x_i + \tilde B(\Xi_i^K)y_i +  q_1 + w^{\nu}_{1i} + r (x_i - x^{\nu}_i) \geq0,  \\
0\leq y_i \perp \tilde M(\Xi_i^K)y_i + \tilde N(\Xi_i^K)x_i  + \tilde q_2(\Xi_i^K) +r (y_i - y^{\nu}_i)  \geq0,
\end{array}\right.
\end{equation}
and obtain a solution $(\hat{x}^{\nu}_i, \hat{y}^{\nu}_i), i\in \bar{K}$.

\textbf{Step 2.}
For $i\in \bar{K}$, let
$$
\bar{x}^{\nu+1}= \sum_{i=1}^K p_i \hat{x}^{\nu}_i, \;\;\,\,
x^{\nu+1}_i=\bar{x}^{\nu+1},\,\, \,
\;\; y^{\nu+1}_i =  \hat{y}^{\nu}_i, \,\,\,\,\, w_{1i}^{\nu+1}  = w_{1i}^{\nu} + r (\hat{x}^{\nu}_i - x^{\nu+1}_i). $$
Set $\nu= \nu+1$, go to Step 1.

\end{alg}

Step 1 solves 
the SLCP per scenario and Step 2 corrects the $x$-component by averaging the obtained scenario based solutions. The $\tilde w_1$ components
serve as auxiliary variables which correspond to multipliers in  PHM \cite{rw1991}.

\section{Distributionally robust formulation of two-stage SLCP}

In this section, we revisit the
two-stage SLCP (\ref{eq:slcp-two1}) by
considering a situation where
 the true probability distribution $P$ is unknown but it is possible to use partial information such as empirical data, computer simulation or  subjective judgements
 to construct an ambiguity set of distributions which contains the true distribution with certain confidence. In such circumstances, it might be sensible to consider
 a robust solution to the two-stage DRLCP \eqref{eq:slcp-two1-1-DRO}.
 In the case that $\CP$ encompasses all probability measures in the support set of $\xi$,
i.e., $\mathscr{P}(\Xi)$, the first stage DRLPC
\eqref{eq:slcp-two1-1-DRO} reduces to
\bgeqn\label{eq:slcpro}
0\leq x \perp A x + B(\xi)y(\xi) +  q_1\geq0, \;\; \forall \xi\in\Xi.
\edeqn
The  system
might not have a solution and consequently one may consider the
 ERM model by replacing the complementarity system with
$$
\min_{x,y(\cdot)} \bbe_P[\|\min(x,A x + B(\xi)y(\xi) +  q_1)\|^2],
$$
where $P$ is any continuous distribution with support set $\Xi$,
see \cite{XF05}.
Our focus here is that $\CP$ is only a subset of $\mathscr{P}(\Xi)$.
We make a blanket assumption that the  two-stage DRLCP \eqref{eq:slcp-two1-1-DRO} has a solution  and discuss 
computational schemes for solving
the problem.
To this end, we write the first stage of (\ref{eq:slcp-two1-1-DRO}) equivalently as
\begin{equation}\label{eq:slcp-two1-1-DRO-r}
\left\{
\begin{array}{ll}
-x \leq  0,\\
- A x - \bbe_P[B(\xi)y(\xi)] - q_1\leq 0, \;\; &\forall P\in \CP,\\
x^T(A x + \bbe_P[B(\xi)y(\xi)] + q_1) \leq 0, \;\; &\forall P\in \CP.
\end{array}
\right.
\end{equation}
It is easy to
verify that the system of inequalities above can be equivalently written as
\begin{equation}
\left\{
\begin{array}{ll}
-x \leq  0,\\
\max_{P\in \CP} [- A x - \bbe_{P}[B(\xi)y(\xi)] - q_1]_i\leq 0, \;\; &\inmat{for}\; i\in\bar{n},
 \\%\forall P\in \CP,\\
\max_{P\in \CP}x^T  (A x +\bbe_{P}[B(\xi)y(\xi)] + q_1) = 0,&
\end{array}
\right.
\label{eq:slcp-two1-1-DRO-r-1-b0}
\end{equation}
where we write $[a]_i$ for the $i$-th component of vector $a$.
Observe that under the first and the second equations of \eqref{eq:slcp-two1-1-DRO-r-1-b0},  the third equation of \eqref{eq:slcp-two1-1-DRO-r-1-b0}
is equivalent to
\bgeqn\label{eq:pi}
\sum_{i=1}^n[x]_i \max_{P\in \CP} [A x +\bbe_{P}[B(\xi)y(\xi)] + q_1]_i \leq0.
\edeqn
To see the equivalence, we note that \eqref{eq:pi} implies the third equation of \eqref{eq:slcp-two1-1-DRO-r-1-b0}.
Conversely, by \eqref{eq:slcp-two1-1-DRO-r}, for every $i\in \bar{n}$,
$$
[x]_i [A x +\bbe_{P}[B(\xi)y(\xi)] + q_1]_i\leq0, \;\; \forall P\in \CP,
$$
which implies \eqref{eq:pi}.
Thus system \eqref{eq:slcp-two1-1-DRO-r-1-b0} can be written as
\begin{equation}\label{eq:slcp-two1-1-DRO-r-1}
\left\{
\begin{array}{lll}
-x \leq  0,\\
\displaystyle\max_{P\in \CP} [- A x - \bbe_{P}[B(\xi)y(\xi)] - q_1]_i\leq 0,&  \inmat{for}\; i\in \bar{n},\\
\displaystyle \sum_{i=1}^n[x]_i \max_{P\in \CP} [A x +\bbe_{P}[B(\xi)y(\xi)] + q_1]_i \leq 0. &
\end{array}
\right.
\end{equation}
Note that if $\CP$ is a convex combination of a finite number of known distributions, then
the maximum in $P$ is achieved at the vertices of $\CP$ and consequently, the problem above reduces to the two-stage SLCP \eqref{eq:slcp-two1}.

In what follows, we consider the case when $\CP$ is constructed through moment conditions:
\bgeq
\CP:=\left\{P\in \mathscr{P}:
\begin{array}{ll}
 \bbe_P[\psi_j(\xi)] = b_j & \inmat{for} \; j=1,\cdots, s\\
  \bbe_P[\psi_j(\xi)] \leq b_j& \inmat{for} \; j=s+1,\cdots, t
\end{array}
 \right\},
\edeq
where all  $\psi_i$ are  Lipschitz continuous function of $\xi$ with Lipschitz constants  $L$  defined in Assumption \ref{A:MNXI}.
Our purpose is to
get rid of the maximum operations w.r.t. $P$ in (\ref{eq:slcp-two1-1-DRO-r-1}) when $\CP$ has the specific structure.
To this end, we need to assume  as in the previous sections
that the second stage of DRLCP  (\ref{eq:slcp-two1-1-DRO})
defines a unique solution $\bar{y}(x,\xi)$ for each fixed $x$ and $\xi$.
Substituting $\bar{y}(x,\xi)$ to the second equation of (\ref{eq:slcp-two1-1-DRO-r-1}), we obtain
$$
\max_{P\in \CP} [- A x - \bbe_P[B(\xi)\bar{y}(x,\xi)] - q_1]_i\leq 0, \; \,\,\inmat{for}\; i\in\bar{n}.
$$
We also need to consider
$$
\max_{P\in \CP} [ A x + \bbe_P[B(\xi)\bar{y}(x,\xi)] + q_1]_i, \;\,\, \inmat{for}\; i\in\bar{n}
$$
in the third equation of (\ref{eq:slcp-two1-1-DRO-r-1}).
To ease  the exposition, we consider
\bgeqn
\max_{P\in\CP} \bbe_P[f_i(x,\xi)], \;\; \inmat{for} \; i\in \bar{n},
\label{eq:DRO}
\edeqn
where  $f_i(x,\xi)$ represents the $i$-th component of
$- A x - B(\xi)\bar{y}(x,\xi) - q_1$ when $i\leq n$ and the $(i-n)$-th component of $Ax + B(\xi)\bar{y}(x, \xi) + q_1$ when $i>n$. 
Define the Lagrange function
\bgeqn
L_i(x,\bm{\Lambda}_i,P) :=\int_\Xi f_i(x,\xi)P(d\xi) + \lambda_0 \left(1-\int_\Xi P(d\xi)\right)
+ \sum_{j=1}^t \lambda^i_j \left(b_j-\int_\Xi \psi_j(\xi)P(d\xi)\right),
\edeqn
and
$$
\bar{\Lambda}_i :=\{\lambda^i = (\lambda^i_0,\lambda^i_1,\cdots,\lambda^i_t)^T: \lambda^i_j\geq 0, \;\inmat{for} \; j=s+1,\cdots,t\}, \;\; i\in \bar{n}.
$$
The Lagrange dual of problem (\ref{eq:DRO}) can be written as
\bgeqn
\min_{\lambda^i\in\bar{\Lambda}_i} \max_{P \in \mathscr{M}} L_i(x,\lambda^i,P),
\label{eq:dual}
\edeqn
where $\mathscr{M}$ denotes the set of all positive measures over $\Xi$. 
Conditions for strong duality can be easily established. For instance, we can consider the
following
Slater type condition
\bgeqn
\label{eq:scq-clas-ex}
(1,0_s,0_{t-s})\in\inmat{int}\{(\langle P,1\rangle,\langle P,\psi_E\rangle,\langle P,\psi_I\rangle)+\K_1:P\in\mathscr{M}_+ \},
\edeqn
 where $\K_1:=\{0\}\times\{0_s\}\times\R_{+}^{t-s}$ and $0_{s}$ denotes the zero vector in $\R^s$, $\psi_E= (\psi_1,\cdots,\psi_s)$
and $\psi_I= (\psi_{s+1},\cdots,\psi_t)$. The following proposition
comes straightforwardly from Xu, Liu and Sun \cite[Proposition 2.1]{XLS17}.

\begin{prop} The following assertions hold.
\begin{itemize}

\item[(i)] Condition (\ref{eq:scq-clas-ex}) is equivalent to
\bgeqn
\label{eq:scq-clas-ex1}
(\mu_E,\mu_I)\in\inmat{int}\{(\langle P,\psi_E\rangle,\langle P,\psi_I\rangle)+\K_2:P\in\mathscr{P}(\Xi) \},
\edeqn
where $\K_2:=\{0_{s}\}\times\R_{+}^{t-s}$.

\item[(ii)] Condition (\ref{eq:scq-clas-ex1}) is fulfilled if
\bgeqn
\mu_E\in \inmat{int} \; \{ \langle P, \psi_E(\xi)\rangle: P\in \mathscr{P}(\Xi)\}
\label{eq:slater-eq-XLS}
\edeqn
and there exists $P_E\in \mathscr{P}(\Xi)$ with $\langle P, \psi_E(\xi)\rangle = \mu_E$
such that
\bgeqn
0_{s-t} \in \inmat{int} \; \{ \langle P_E, \psi_I(\xi)\rangle -\mu_I -\R_-^{s-t}\}.
\label{eq:Slater-class-example-a}
\edeqn
In the case when $s=t$, i.e., there is no inequality constraint, condition (\ref{eq:slater-eq-XLS}) coincides
with condition (\ref{eq:scq-clas-ex1}). Likewise, when $s=0$, i.e., there is no
equality constraint, (\ref{eq:Slater-class-example-a})
reduces to existence of $P\in \mathscr{P}(\Xi)$ such that
\bgeqn
0_{t} \in \inmat{int} \; \{ \langle P, \psi_I(\xi)\rangle -\mu_I -\R_-^{t}\},
\label{eq:Slater-class-example-a}
\edeqn
which coincides with (\ref{eq:scq-clas-ex1}).

\item[(iii)] Condition (\ref{eq:slater-eq-XLS}) holds
naturally in the case when
\bgeqn
\{ \langle P, \psi_E(\xi)\rangle: P\in \mathscr{P}(\Xi)\}=\R^s
\label{eq:slater-eq-fullspace}
\edeqn
whereas
condition (\ref{eq:Slater-class-example-a}) is fulfilled
if there exists $P_E\in \mathscr{P}(\Xi)$ with $\langle P, \psi_E(\xi)\rangle = \mu_E$ such that
 \bgeqn
 \langle P_E, \psi_I(\xi)\rangle -\mu_I<0.
 \label{eq:SCQ-st}
 \edeqn
\end{itemize}
\label{p-classical-moments}
\end{prop}
Throughout this section, we make a blanket assumption that the strong duality holds.

We now return to discuss (\ref{eq:dual}).
Through standard analysis of Lagrange duality (see i.e. discussions at page 308 of \cite{SP09}),
we have
\bgeqn
\max_{P \in \mathscr{M}} L_i(x,\lambda^i,P) 
=  \sum_{j=1}^t \lambda^i_j b_j + \max_{\xi \in \Xi} \left(f_i(x,\xi) -\sum_{j=1}^t \lambda^i_j \psi_j(\xi)\right).
\label{eq:H-minimax}
\edeqn
Consequently, \eqref{eq:slcp-two1-1-DRO-r-1} can be written as
\bgeqn
\label{eq:slcp-two1-1-DRO-r-3-i}
\left\{
\begin{array}{ll}
-x \leq  0, \;
\bm{\Lambda}_1, \bm{\Lambda}_2\in\bar{\Lambda}, &
\\
\left[-Ax-B(\xi)y(\xi) -q_1\right]_i
 - \bm{\Lambda}_1(\psi(\xi)-b)
\leq 0, \; & \forall i\in \bar{n},\; \xi\in \Xi,\\
x^T\left[ (Ax+B(\xi)y(\xi) +q_1)  -  \bm{\Lambda}_2(\psi(\xi) - b)\right]
\leq 0, \;\; & \forall \xi\in \Xi, 
\\
0\leq y(\xi) \perp M(\xi)y(\xi) + N(\xi)x  + q_2(\xi)\geq0, \;\; & \forall \xi\in\Xi,
\end{array}
\right.
\edeqn
where
$
 \bar{\Lambda} :=\{\bm{\Lambda} = (\lambda^1, \cdots, \lambda^n)^T \in \R^{n\times t}: -\lambda_j\leq 0, \inmat{for}\; j=s+1,\cdots,t\}.
 $

We consider the discrete approximation of \eqref{eq:slcp-two1-1-DRO-r-3-i}.
As what we did for the two-stage SLCP in Section 3, our first step is to
develop a discretize approximation  of the infinite inequality system.
At this point, it might be helpful to point out the difference between the two-stage SLCP
and (\ref{eq:slcp-two1-1-DRO-r-3-i}): the former depends on the true probability distribution of
$\xi$ whereas the latter is independent of the probability distribution and it is entirely
determined by the support set $\Xi$. This motivates us to adopt a slightly different discretization approach
by using Monte Carlo sampling.

Let $\{\xi^i\}_{i=1}^K$ be i.i.d samples of $\xi$ generated by
any probability distribution with support set $\Xi$.
We
consider the following
discretization scheme for \eqref{eq:slcp-two1-1-DRO-r-3-i}:
\bgeqn
\label{eq:slcp-two1-1-DRO-r-discrete}
\left\{
\begin{array}{ll}
-x\leq0, \bm{\Lambda}_1, \bm{\Lambda}_2\in \bar{\Lambda}, &\\
 \left(-Ax - B(\xi^i){\bf y}_i -q_1\right) -
\bm{\Lambda}_1(\psi(\xi^i)-b)
\leq 0, \; & i\in \bar{K},\\
 x ^T \left[Ax + B(\xi^i){\bf y}_i +q_1 -
\bm{\Lambda}_2(\psi(\xi^i)-b)\right]
\leq 0, \; & i\in \bar{K},\\
0\leq {\bf y}_i \perp M(\xi^i){\bf y}_i + N(\xi^i)x  + q_2(\xi^i)\geq0, \;\; & i\in \bar{K}.
\end{array}
\right.
\edeqn
Here ${\bf y}_i$ is determined by the second stage LCP at sampled point $\xi^i$.
This is in contrast to (\ref{eq:cpki}) where ${\bf y}_i$ is determined
by the average value of the second stage problem data over set $\Xi_i^K$.
The underlying reason is that the first stage inequality system (the first three inequalities of (\ref{eq:slcp-two1-1-DRO-r-discrete})) here
does not involve any probability distribution.
Of course, (\ref{eq:slcp-two1-1-DRO-r-discrete})
depends on the iid samples and hence the probability distribution which
generates them.
We should also point out that
problem (\ref{eq:slcp-two1-1-DRO-r-discrete}) is a dual formulation of
(\ref{eq:slcp-two1-1-DRO}) with the ambiguity
\bgeq
\CP_K:=\left\{P\in \mathscr{P}(\Xi_K):
\begin{array}{ll}
 \sum_{i=1}^K \psi_j(\xi^i) = b_j & \inmat{for} \; j=1,\cdots, s\\
 \sum_{i=1}^K \psi_j(\xi^i)\leq b_j& \inmat{for} \; j=s+1,\cdots, t
\end{array}
 \right\},
\edeq
where $\Xi_K:=\{\xi^1,\cdots,\xi^K\}$.

Problem (\ref{eq:slcp-two1-1-DRO-r-discrete}) comprises
$Km$
complementarity problems and $K(2n+1)$ inequalities with $(1+2t)n+Km$ variables.
It is easy to see that a solution to the true problem (\ref{eq:slcp-two1-1-DRO-r-3-i}) is also a solution to the discretized problem (\ref{eq:slcp-two1-1-DRO-r-discrete}).
In what follows, we analyze convergence of the latter as $K\to \infty$.

\begin{thm}
  Let $\{(x^K, \bm{\Lambda}_1^K, \bm{\Lambda}_2^K, {\bf y}^K)\}$ be a sequence of
   solutions of
   \eqref{eq:slcp-two1-1-DRO-r-discrete} with different
   size of samples.
Assume: (a)  $\Xi$ is a compact set,
(b) Assumption \ref{A:MNXI} holds, and
(c) the iid samples
$\xi^1,\cdots,\xi^K$  are generated by randomizing $\xi$
and attaching to it with a continuous  probability distribution
$P$ over $\Xi$ such that
\[
P(\|\xi-\xi_0\|\leq \delta) > C \delta^\nu
\]
for any fixed point $\xi_0\in \Xi$  and $\delta\in (0,\delta_0)$, where $C$, $\nu$ and $\delta_0$  are some positive constants.
Then every
 cluster point of
 the sequence
 $\{(x^K, \bm{\Lambda}_1^K, \bm{\Lambda}_2^K)\}$
 is the solution of  \eqref{eq:slcp-two1-1-DRO} w.p.1.
\end{thm}
\begin{proof}
By taking a subsequence if necessary, we assume for the simplicity of notation that
 $\{(x^K, \bm{\Lambda}_1^K, \bm{\Lambda}_2^K)\}$ converges to  $(\hat{x}, \hat{\bm{\Lambda}}_1, \hat{\bm{\Lambda}}_2)$.
This means that there exists a positive number $\rho$ such that
  $\|(x^K, \bm{\Lambda}_1^K, \bm{\Lambda}_2^K)\| \leq \rho$ for all $K$.
Let $\Xi_1^K,\cdots, \Xi_K^K$ be the Voronoi partition of $\Xi$ centred at
$\xi^1,\cdots,\xi^K$.
Thus for every $\xi\in \Xi$, there exists a Voronoi cell $\Xi^i_K$ centred at
$\xi^{i_K}$ such that $\xi\in \Xi_i^K$.
Moreover, under condition (c), we can easily use
\cite[Lemma 3.1]{XLS17} to show that
$\max_{j\in\bar{K}}\Delta_j^K\to 0$ at exponential rate as $K\to \infty $. See \cite[Proposition 8]{LPX17}.

On the other hand, under condition (b), we can show,
following  a similar analysis to Step 1 in the proof of Theorem \ref{t-quantitative-sect3} ,
 that there exists a positive constant $C$ such that
\bgeqn
\|\bar{y}(x, \xi^{i_K}) - \bar{y}(x, \xi)\|
\leq  CL
\Delta(\Xi_i^K).
\edeqn
Let
$$
G(x, \bm{\Lambda}_1, \bm{\Lambda}_2, \xi)  = \begin{pmatrix}
- Ax - B(\xi)\bar{y}(x,\xi) -q_1 -
\bm{\Lambda}_1(\psi(\xi)-b) \\
x^T( Ax + B(\xi)\bar{y}(x,\xi) + q_1 -
\bm{\Lambda}_2(\psi(\xi)-b))
\end{pmatrix},
$$
$$
F_j(x, \bm{\Lambda}_1, \bm{\Lambda}_2)= \max_{\xi\in \Xi}
G_j(x, \bm{\Lambda}_1, \bm{\Lambda}_2, \xi),
$$
and
$$
F_j^K(x, \bm{\Lambda}_1, \bm{\Lambda}_2)= \max_{i\in \bar{K}}G_j(x, \bm{\Lambda}_1, \bm{\Lambda}_2, \xi^i), \;\,\,\,\inmat{for}\;
 j\in \overline{n+1}.
$$
Then we can write the first three equations in
 (\ref{eq:slcp-two1-1-DRO-r-3-i})
 equivalently as
\bgeqn
\label{eq:slcp-two1-1-DRO-r-3-ineq}
\left\{
\begin{array}{l}
\bm{\Lambda}_1, \bm{\Lambda}_2\in \bar{\Lambda},
\\
x \geq 0,\\
F_j(x, \bm{\Lambda}_1, \bm{\Lambda}_2) \leq 0,\,\, \;\inmat{for}\;
 j\in \overline{n+1},
\end{array}
\right.
\edeqn
and
\bgeqn
\label{eq:slcp-two1-1-DRO-r-3-ineq-dis}
\left\{
\begin{array}{l}
\bm{\Lambda}_1, \bm{\Lambda}_2\in \bar{\Lambda},
\\
x\geq 0,\\
F_j^K(x, \bm{\Lambda}_1, \bm{\Lambda}_2) \leq 0, \;\,\,\inmat{for}\;
 j\in \overline{n+1}.
\end{array}
\right.
\edeqn
Let
$$
G^K(x, \bm{\Lambda}_1, \bm{\Lambda}_2, \xi) = \sum_{i=1}^K\mathbf{1}_{\Xi_i^K}(\xi)G(x, \bm{\Lambda}_1, \bm{\Lambda}_2, \xi^i).
$$
Then $F_j^K(x, \bm{\Lambda}_1, \bm{\Lambda}_2) = \max_{\xi\in\Xi} G_j^K(x, \bm{\Lambda}_1, \bm{\Lambda}_2, \xi)$.
Under Assumption \ref{A:MNXI}, for any $(x,\bm{\Lambda}_1, \bm{\Lambda}_2)$ such that $\|(x,\bm{\Lambda}_1, \bm{\Lambda}_2)\|\leq \rho$, we have the boundedness of $\bar{y}(x,\xi)$ and then there exists positive constant $C_2$ such that
$$
\begin{array}{lll}
 &&\displaystyle{\max_{i\in\overline{n+1}}}|F_i(x, \bm{\Lambda}_1, \bm{\Lambda}_2) - F_i^K(x, \bm{\Lambda}_1, \bm{\Lambda}_2)|\\
 && \leq    \displaystyle{\max_{i\in\overline{n+1}}\max_{\xi\in \Xi}}  |G_i(x, \bm{\Lambda}_1, \bm{\Lambda}_2, \xi) - G_i^K(x, \bm{\Lambda}_1, \bm{\Lambda}_2, \xi) | \\
 && \leq   \displaystyle{\max_{ i\in\bar{K}, k=1,2} \max_{\xi\in\Xi_{i}^K}}(\rho+1) \|B(\xi)\bar{y}(x,\xi) + \bm{\Lambda}_k \psi(\xi) - B(\xi^k)\bar{y}(x,\xi^i) - \bm{\Lambda}_k \psi(\xi^i) \|\\
&& \leq   (\rho+1)LC_2\Delta(\Xi_i^K).
\end{array}
$$
This shows
$$
 \lim_{K\to\infty} \displaystyle{\sup_{\|(x,\bm{\Lambda}_1, \bm{\Lambda}_2)\|\leq \rho} \max_{k\in\overline{n+1}}}|F_k(x, \bm{\Lambda}_1, \bm{\Lambda}_2) - F_k^K(x, \bm{\Lambda}_1, \bm{\Lambda}_2)| =0,  \;\;\inmat{w.p.1}.
$$
By \cite[Lemma 4.2 (i)]{Xu10a},  any cluster point of the sequence of solutions
$\{(x^K, \bm{\Lambda}_1^K, \bm{\Lambda}_2^K)\}$ obtained from solving  system
(\ref{eq:slcp-two1-1-DRO-r-3-ineq-dis})
is a  solution of \eqref{eq:slcp-two1-1-DRO-r-3-ineq} almost surely.
\hfill$\Box$

\end{proof}

\section{Two-stage  distributionally robust game}

We consider a duopoly market\footnote{The model can be easily extended to an oligopoly, we consider
a duopoly for simplicity of exposition so that we can concentrate on the main ideas.} where two firms compete to supply a
homogeneous product (or service) noncooperatively in future.
Neither of the firms has an existing capacity and thus must
make a decision at the present time on their capacity
for future
supply of quantities in order to allow themselves enough time to
build the necessary facilities.

 The market demand in future is characterized by a random inverse
demand function $p(q,\xi(\omega))$, where $p(q,\xi(\omega))$ is
the market price, $q$ is the total supply to the market, and
$\xi:\Omega\rightarrow\R$ is a continuous random variable.
Specifically, for each realization of the random variable
$\xi:\Omega\rightarrow\R$,  we obtain a different inverse demand
function $p(q,\xi(\omega))$. The uncertainty in the inverse demand
function is then characterized by the distribution of the random
variable $\xi$.

Firm $i$'s cost function for building up capacity $x_i$ is
 $C_i(x_i)$ and the cost of producing (supplying) a quantity of $y_i$ in future is
  $H_i(y_i, \xi)$,  $i=1,2$. Assuming each firm aims to maximize the expected profit,
  we can then develop a mathematical model for their decision making: for $i=1,2$, find $(x^*_i, y^*_i(\cdot))$
  such that it solves the following two-stage stochastic programming problem
\begin{eqnarray}
\begin{array}{cl}
\displaystyle \max_{x_i,y_i(\cdot)}  & \bbe_P[p(y_i(\xi)+y^*_{-i}(\xi),\xi)y_i(\xi) -H_i(y_i(\xi), \xi)] - C_i(x_i)   \\
\mbox{s.t.}                    &0\leq y_i(\xi) \leq x_i,
\end{array}
\label{eq:SNash}
\end{eqnarray}
where the mathematical expectation is taken w.r.t. the distribution of $\xi$ and by convention we write $y_{-i}$
for decision variable of the firm(s) other than $i$.
This is a closed loop two-stage stochastic Nash-Cournot game where each player (firm) needs to make a decision
on capacity   before realization of uncertainty
anticipating  competition in future (second stage).
At this point, we refer readers to Wongrin et al. \cite{WHRCB13}
for a deterministic model with application in electricity markets,
and a more sophisticated two-stage stochastic model by Luna, Sagastiz\'abal and Solodov \cite{LSS}
 where each player is risk-averse and all players share an identical constraint in the second stage.  Similar models  can also be found in Ralph and Smeers \cite{RaS11} for stochastic endogenous equilibrium in asset pricing.
Here we concentrate on reformulation of
problem (\ref{eq:SNash}) as
a two-stage SLCP under some moderate conditions  and investigate
the latter under this particular context.

Let us now consider a situation where each player does not have complete information on the true  probability distribution  $P$. However, each player may use available partial information to construct an ambiguity set of probability distributions, denoted respectively by $\CP_1,\CP_2$. Assuming both players base their decision on the worst probability distribution, then we may consider a distributionally robust game: for $i=1,2$, find $(x_i^*, y_i^*(\cdot))$ such that
\begin{eqnarray}
\begin{array}{lcl}
(x_i^*, y_i^*(\cdot))\in&\displaystyle \arg\max_{x_i,y_i(\cdot)} \min_{P_i \in \CP_i}& \bbe_{P_i}[p(y_i(\xi)+y^*_{-i}(\xi),\xi)y_i(\xi) -H_i(y_i(\xi), \xi)] - C_i(x_i)   \\
&\mbox{s.t.}                    &0\leq y_i(\xi) \leq x_i.
\end{array}
\label{eq:SNash-dro}
\end{eqnarray}
To see the structure of \eqref{eq:SNash-dro} clearly, we write down the optimal decision making problems at the second stage
game after the market demand is observed by both firms, that is,
 for $i=1,2$, find $y^*_i$ such that
\begin{eqnarray}
\begin{array}{lcl}
y_i^*(\xi)\in &\displaystyle \arg\max_{y_i}  & p(y_i+y^*_{-i},\xi)y_i -H_i(y_i, \xi)  \\
&\mbox{s.t.}                    &0\leq y_i \leq x_i,
\end{array}
\label{eq:SNash-second-stg}
\end{eqnarray}
where $x_i$ is fixed. To analyse \eqref{eq:SNash-second-stg}, we need to make some assumption on the cost functions $H_i(y_i, \xi)$ and the inverse demand function $p(q, \xi)$.

\begin{assu} \rm For $i=1,2$, $H_i(y_i, \xi)$ is twice continuously differentiable,
$H_i'(y_i, \xi)\geq 0$ and $H_i''(y_i, \xi)\geq 0$ for $y_i\geq 0$.
\label{acost}
\end{assu}
This assumption is standard. It requires that the production cost function of each firm be
convex and sufficiently smooth, see \cite{Xu05} and references therein.

\begin{assu} \rm The inverse demand function $p(q,\xi)$ satisfies the following conditions.
\begin{itemize}
\item[(i)] $p(q,\xi)$ is twice continuously differentiable
in $q$ and $p_q'(q,\xi)<0$ for $q\geq 0$ and $\xi\in
\Xi$.

\item[(ii)] $p'_q(q,\xi)+qp''_{qq}(q,\xi)\leq 0$, for $q\geq 0$ and $\xi\in
\Xi$.
\end{itemize}
\label{A1}
\end{assu}
This assumption is similar to an assumption used by Sherali,
Soyster and Murphy \cite{mss2} and  De Wolf and Smeers \cite{ws}.
Consider a monopoly market with an extraneous supply $\bar{c}\geq 0$. If
the monopoly's output is $q$, then its revenue at demand scenario
$\epsilon(\omega)=\xi$ is $q(p(q+\bar{c},\xi))$. The marginal revenue is
$p(q+\bar{c},\xi)+qp'_q(q+\bar{c},\xi)$. The rate of change of this marginal
revenue with respect to the increase in the extraneous supply $\bar{c}$
is $p'_q(q+\bar{c},\xi)+qp''_{qq}(q+\bar{c},\xi)$. Assumption \ref{A1} (ii)
implies that this rate is  not positive when $\bar{c}=0$ for any $\xi\in
\Xi$. In other words, any extraneous supply will potentially
reduce the monopoly's marginal revenue in any demand scenario. See
\cite{mss2} for a similar explanation for a deterministic
leader-followers' market. The following result is established by Xu \cite{Xu05}.

\begin{prop} Under Assumption \ref{A1}, the following assertions hold.
\begin{itemize}
\item[(i)] For fixed $\bar{c}\geq 0$,
\bgeqn
p'_q(q+\bar{c},\xi)+qp''_{qq}(q+\bar{c},\xi)\leq 0, \; \mbox{for} \; q\geq 0, \; \xi\in \Xi.
\label{eqpK}
\edeqn
\item[(ii)] $qp(q+\bar{c},\xi)$ is strictly
concave in $q$ for $q\geq 0$, $\xi\in \Xi$.
\end{itemize}
\label{pconcave}
\end{prop}

By Proposition \ref{acost}, we know that  problem (\ref{eq:SNash-second-stg}) has a unique optimization solution (we are short of claiming unique equilibrium at this point) for each $i$. 
Moreover, we can write down the Karush-Kuhn-Tucker (KKT) conditions
for \eqref{eq:SNash-second-stg} as follows:
\begin{equation}
0\leq \begin{pmatrix}
y_i\\
\mu_i
\end{pmatrix} \perp
\begin{pmatrix}
-p(y_i+y_{-i},\xi)- y_ip_q'(y_i+y_{-i},\xi)+H_i'(y_i, \xi)  + \mu_i\\
x_i - y_i
\end{pmatrix}\geq0,
\label{eq:KKT-1}
\end{equation}
where  $\mu_i$, $i=1,2$, are Lagrange multipliers of constraints $y_i(\xi)\leq x_i$, $i=1,2$.
Moreover, by Rosen \cite[Theorem 1]{rosen},
 the second stage Nash-Cournot game has an equilibrium which means
 the second stage complementarity problem (\ref{eq:KKT-1}) has a solution. The solution depends on $x_1,x_2$ and $\xi$, we denote it by $\bar{y}(x,\xi)$ and write $x$ for $(x_1, x_2)$.

With the second stage equilibrium $\bar{y}(x,\xi)$, we are ready to write down the first stage
decision making problem for player $i$:
\begin{eqnarray}
\begin{array}{cl}
\displaystyle \max_{x_i,y_i(\cdot)} \min_{P_i \in \CP_i}& \bbe_{P_i}[
v_i(x,\xi)]
  - C_i(x_i)   \\
\mbox{s.t.}                    &
x_i \geq 0,
\end{array}
\label{eq:SNash-dro-stage1}
\end{eqnarray}
where
$$
v_i(x,\xi) := p(\bar{y}_i(x,\xi)+\bar{y}_{-i}(x,\xi),\xi)\bar{y}_i(x,\xi) -H_i(\bar{y}_i(x,\xi), \xi).
$$
A $4$-tuple $(x_1^*,x_2^*, y_1^*(\cdot), y_2^*(\cdot))$ with $(y_1^*(\cdot), y_2^*(\cdot)) = (\bar{y}_1(x^*, \cdot), \bar{y}_2(x^*, \cdot))$ is called a
{\em two-stage distributionally robust equilibrium}
 if  $(x_i^*,x_{-i}^*)$ solves (\ref{eq:SNash-dro-stage1}).
Following Agassi and Bertsimas \cite{AgB06}, we may introduce so-called
{\em ex post equilibrium} $(x^*_i, x^*_{-i})$ which satisfies
\bgeqn\label{eq:vc}
x_i^* \in \arg\max \bbe_{P_i}[
v_i(x_i,x_{-i}^*,\xi)]- C_i(x_i), \quad \forall P_i \in\CP_i, \;\; i=1,2.
\edeqn
It is easy to prove that any ex post equilibrium is a distributionally robust equilibrium.
Assuming that $C_i(x_i)$ is continuously differentiable, we may write down the first order optimality condition
of \eqref{eq:vc}:
\bgeqn
0\in \bbe_{P_i}[
\partial_{x_i} v_i(x,\xi)]- C_i'(x_i) + {\cal N}_{[0,\infty)}(x_i), \quad \forall
P_i \in\CP_i, i=1,2,
\label{eq:stage-Vi}
\edeqn
where ${\cal N}_{[0,\infty)}(x_i)$ denotes the normal cone of interval $[0,+\infty)$ at $x_i$ and
$\partial v_i$ denotes the Clarke subdifferential  of $v_i$ with respect to $x_i$.
Note that under Assumption \ref{A1}, the second stage problem (\ref{eq:SNash-second-stg})
has a unique solution and the set of Lagrange multipliers defined by the
KKT system (\ref{eq:KKT-1}) is a singleton,
it follows by Ralph and Xu \cite[Lemma 5.2]{RaX11} that
$v(x,\xi)$ is continuously differentiable w.r.t. $x_i$ for $x_i>0$ and
\bgeq
\nabla_{x_i}v_i(x,\xi) = \frac{L(y_i(\xi),\lambda_i(\xi),\mu_i(\xi),\x_i)}{dx_i}= \mu_i(\xi),
\edeq
where
$$
L(y_i(\xi),\lambda_i(\xi),\mu_i(\xi),\x_i) := p(y_1(\xi)+y_2(\xi),\xi)y_i(\xi) -H_i(y_i(\xi), \xi) +\lambda_i(\xi)y_i(\xi) - \mu_i(\xi)(y_i(\xi)- x_i).
$$
Consequently,  we can rewrite \eqref{eq:stage-Vi} as
\bgeqn
0\leq x_i\perp \bbe_{P_i}[\mu_i(\xi)]- C_i'(x_i) \geq 0,
 \quad \forall
P_i \in\CP_i,\,\, i=1,2.
\label{eq:stage-NCP}
\edeqn
Let
 $$
 g_i(y_i,y_{-i},\xi) = -p(y_i+y_{-i},\xi)- y_ip_q'(y_i+y_{-i},\xi)+H_i'(y_i, \xi), \;\;  \inmat{for} \;\; i=1,2.
 $$
 Note that in the case when $x_i=0$, $y_i(\xi)\equiv 0$, we have 
$$
\partial_{x_i}v_i(x,\xi)
=\{\mu_i(\xi):  \mu_i(\xi) \geq (-g_i(0, 0,\xi))_+, \forall \xi\in \Xi \;\;\inmat{ and }\;\; \bbe_{P_i}[\mu_i(\xi)] \geq  C_i'(x_i), \; \forall P_i\in \CP_i\}.
$$
Summarizing the discussions above, we can derive the following two-stage ex post complementarity problem
\bgeqn\label{eq:twogame}
\left\{
\begin{array}{ll}
0\leq x_i\perp \bbe_{P_i}[\mu_i(\xi)]- C_i'(x_i) \geq 0,
 \quad &\forall
P_i \in\CP_i, \;\; i=1,2,
\\
0\leq \begin{pmatrix}
y_i(\xi)\\
\mu_i(\xi)
\end{pmatrix} \perp
\begin{pmatrix}
g_i(y_i(\xi),y_{-i}(\xi),\xi)+ \mu_i(\xi)\\
x_i - y_i(\xi)
\end{pmatrix}\geq 0, \;\; &\inmat{for} \;\;P_i\inmat{-a.e.}\;\; \xi\in \Xi,  \;\; i=1,2.
\end{array}
\right.
\edeqn

 \subsection{An example}

To explain how to reformulate a  two-stage duopoly game as a
two-stage ex post complementarity problem (\ref{eq:twogame}), we consider a simple example
where $ p(q, \xi) = a(\xi_1) - b(\xi_1)q + \xi_2,$
 $$
 C_i(x_i) = \alpha_i + \beta_ix_i- \frac{1}{2}\gamma_ix_i^2 \;\;\;\inmat{and}\; H_i(y_i,\xi_1) = s_i(\xi_1)+ \zeta_i(\xi_1) y_i+ \frac{1}{2}\eta_i(\xi_1) y_i^2 \;\;\inmat{for} \;\; i=1,2,
 $$
  where $\xi = (\xi_1, \xi_2)^T$ is a  random vector with support set  $\Xi: = [-1, 1]\times [-1,1]$,  $\xi_1, \xi_2$ are independent,  $a(\xi_1)$, $b(\xi_1)$, $s_i(\xi_1)$, $\zeta_i(\xi_1)$ and $\eta_i(\xi_1)$ map from $[-1,1]$ to $\R_{++}$,  and
 $\alpha_i$, $\beta_i$, $\zeta_i$, $\eta_i$, and $\gamma_i$ are positive constants. The cost functions $C_i(x_i)$, $i=1,2$, are concave which means the marginal cost for capacity set-up
 decreases for both players as capacity increases.
 We assume that the ratio $\beta_i/\gamma_i$ is large
 enough
 so that the marginal capital cost
 does not become negative within any possible capacity that the two players may install.
  To ease the exposition, let
 $$
 A= \begin{pmatrix}
\gamma_1 & 0\\
 0 & \gamma_2
\end{pmatrix}, \;\;  \;\;  B= \begin{pmatrix}
0 & 0 & 1 &0\\
0 & 0 & 0 & 1
\end{pmatrix}, \;\;  \;\;  N= B^T,
$$
$$
\Pi(\xi_1)= \begin{pmatrix}
2b(\xi_1) + \eta_1(\xi_1) & b(\xi_1) \\
b(\xi_1) & 2b(\xi_1) + \eta_2(\xi_1)
\end{pmatrix}
\; \inmat{and} \;\;
M(\xi_1)= \begin{pmatrix}
\Pi(\xi_1) & I_2 \\
-I_2 & 0
\end{pmatrix}.
 $$
 We can write \eqref{eq:twogame}
 in the following matrix-vector form:
\bgeqn {\small
\left\{
\begin{array}{l}
0\leq \begin{pmatrix}
x_1\\
x_2
\end{pmatrix}\perp
A
\begin{pmatrix}
x_1\\
x_2
\end{pmatrix} +
\bbe_{P}\left[
B
\begin{pmatrix}
y_1(\xi),
y_2(\xi),
\mu_1(\xi),
\mu_2(\xi)
\end{pmatrix}^T\right]- \begin{pmatrix}
\beta_1\\
\beta_2
\end{pmatrix} \geq 0, \;\; \forall P\in\CP,
\\
\vspace{2mm}
0\leq \begin{pmatrix}
y_1(\xi)\\
y_2(\xi)\\
\mu_1(\xi)\\
\mu_2(\xi)
\end{pmatrix} \perp
M(\xi_1)
 \begin{pmatrix}
y_1(\xi)\\
y_2(\xi)\\
\mu_1(\xi)\\
\mu_2(\xi)
\end{pmatrix} +
N \begin{pmatrix}
x_1\\
x_2
\end{pmatrix} -
\begin{pmatrix}
a(\xi_1)+\xi_2 - \zeta_1(\xi_1)\\
a(\xi_1)+\xi_2- \zeta_2(\xi_1)\\
0\\
0
\end{pmatrix}
 \geq 0, \quad \; \forall \xi\in \Xi.
\end{array}
\right.
\label{eq:stage-NCP-2-DRO}}
\edeqn
Note that matrix $M(\xi_1)$ is positive semidefinite and nonsingular for all $\xi\in\Xi$.

\subsubsection{Stochastic version}
 In the case when $\CP_i$ reduces to a singleton, that is,
  the true probability distribution,  the two-stage distributionally robust game  \eqref{eq:SNash-dro}
   collapses to two-stage stochastic game  \eqref{eq:SNash}  and consequently \eqref{eq:stage-NCP-2-DRO} can be written as a two-stage SLCP,
\bgeqn {\small
\left\{
\begin{array}{l}
0\leq \begin{pmatrix}
x_1\\
x_2
\end{pmatrix}\perp
A
\begin{pmatrix}
x_1\\
x_2
\end{pmatrix} +
\bbe\left[
B
\begin{pmatrix}
y_1(\xi),
y_2(\xi),
\mu_1(\xi),
\mu_2(\xi)
\end{pmatrix}^T\right]- \begin{pmatrix}
\beta_1\\
\beta_2
\end{pmatrix} \geq 0,
\\
\vspace{2mm}
0\leq \begin{pmatrix}
y_1(\xi)\\
y_2(\xi)\\
\mu_1(\xi)\\
\mu_2(\xi)
\end{pmatrix} \perp
M(\xi_1)
 \begin{pmatrix}
y_1(\xi)\\
y_2(\xi)\\
\mu_1(\xi)\\
\mu_2(\xi)
\end{pmatrix} +
N \begin{pmatrix}
x_1\\
x_2
\end{pmatrix} -
\begin{pmatrix}
a(\xi_1)+\xi_2 - \zeta_1(\xi_1)\\
a(\xi_1)+\xi_2-\zeta_2(\xi_1)\\
0\\
0
\end{pmatrix}
 \geq 0, \forall \xi\in \Xi.
\label{eq:stage-NCP-2}
\end{array}
\right.}
\edeqn

By \cite[Theorem 2]{rosen}, the second stage game  \eqref{eq:SNash-second-stg} has a
unique equilibrium at each scenario $\xi\in \Xi$, which means
that the $(y_1(\xi), y_2(\xi))$-components of the solutions to
the second stage complementarity problem is unique for each $\xi$.
The Lagrange multiplies $\mu_i(\xi)$, $i=1,2$ are also unique
when $x_1, x_2>0$.

\subsubsection{DRLCP version}
In this subsection, we consider the case that the true probability distribution  is unknown
and a distributionally robust game described as in (\ref{eq:SNash-dro}) is played with  the ambiguity set being
 defined through moment conditions $
\CP:=\left\{P\in \mathscr{P}:
 \bbe_P[\xi] = 0
 \right\}.
$

For the simplicity of analysis, we assume that both players use the same ambiguity set, that is $\CP_1=\CP_2=\CP$. Moreover,
we
set $\beta_1=\beta_2=\beta$ and $\zeta_1(\xi_1)=\zeta_2(\xi_1)=1$ in \eqref{eq:stage-NCP-2-DRO}.
Observe that \eqref{eq:stage-NCP-2-DRO} may have a trivial solution with
$x_1=x_2=0$, $y_1(\xi)=y_2(\xi)=0$
for all $\xi\in \Xi $
and
\bgeq
\mu_i(\xi)\geq (a(\xi_1)+\xi_2-1)_+, & \forall \xi\in\Xi  \;\; {\rm and} \;\;
\bbe_P[\mu_i(\xi)] \geq \beta, & \forall P\in\CP, \;\;\inmat{for} \;\; i=1,2.
\edeq

In what follows, we concentrate on non-trivial solutions
with both $x_1$ and $x_2$ being positive.
We endeavour to obtain an analytical solution to (\ref{eq:stage-NCP-2-DRO}).
 For this purpose, we simply assume that  $a$, $b$, $s_i$  are  deterministic positive numbers and $\eta_i(\xi_1)=\bar{\eta}_i+\xi_1$, where $\bar{\eta}_i\ge 1$.

By \eqref{eq:slcp-two1-1-DRO-r-3-i}, the dual formulation of
% form of
\eqref{eq:stage-NCP-2-DRO} can be written as
 \bgeqn {\small
\left\{
\begin{array}{l}
x_1, x_2\geq0,\\
\left[-A
\begin{pmatrix}
x_1\\
x_2
\end{pmatrix} -
B
\begin{pmatrix}
y_1(\xi),
y_2(\xi),
\mu_1(\xi),
\mu_2(\xi)
\end{pmatrix}^T+ \begin{pmatrix}
\beta\\
\beta
\end{pmatrix} -\bm{\Lambda}_1
\begin{pmatrix}
\xi_1\\
\xi_2
\end{pmatrix}\right]  \leq 0, \; \forall \xi\in \Xi,\\
 \begin{pmatrix}
x_1\\
x_2
\end{pmatrix}^T
\left[A
\begin{pmatrix}
x_1\\
x_2
\end{pmatrix} +
B
\begin{pmatrix}
y_1(\xi),
y_2(\xi),
\mu_1(\xi),
\mu_2(\xi)
\end{pmatrix}^T- \begin{pmatrix}
\beta\\
\beta
\end{pmatrix} -\bm{\Lambda}_2
\begin{pmatrix}
\xi_1\\
\xi_2
\end{pmatrix}\right]  \leq 0, \; \forall \xi\in \Xi,
\\
\vspace{2mm}
0\leq \begin{pmatrix}
y_1(\xi)\\
y_2(\xi)\\
\mu_1(\xi)\\
\mu_2(\xi)
\end{pmatrix} \perp
M(\xi_1)
 \begin{pmatrix}
y_1(\xi)\\
y_2(\xi)\\
\mu_1(\xi)\\
\mu_2(\xi)
\end{pmatrix} +
N
\begin{pmatrix}
x_1\\
x_2
\end{pmatrix}-
\begin{pmatrix}
a+\xi_2 - 1\\
a+\xi_2-1\\
0\\
0
\end{pmatrix}
 \geq 0, \forall \xi\in \Xi,
\end{array}
\right.
\label{eq:stage-NCP-ex-dro}}
\edeqn
where
$
\bm{\Lambda}_i =
\begin{pmatrix}
\lambda_{11}^i & \lambda_{12}^i\\
\lambda_{21}^i & \lambda_{22}^i
\end{pmatrix}\in \R^{2\times 2}.
$
Since $\Xi$ is a compact set and
the underlying functions are continuous,
problems \eqref{eq:stage-NCP-2-DRO} and \eqref{eq:stage-NCP-ex-dro}
are equivalent in that there is no dual gap in deriving the Lagrange dual of
maximization with respect to $P$, see \cite[page 208]{SP09}.
Let
$
\tilde y(\xi) = (\tilde y_1(\xi), \tilde y_2(\xi))^T = \Pi(\xi_1)^{-1}\begin{pmatrix}  a+\xi_2-1 \\ a+\xi_2-1 \end{pmatrix}$.
Note that $ \Pi(\xi_1)^{-1}$ is positive definite and diagonally dominant, therefore $\tilde y(\xi)>0$ when $a+\xi_2-1>0$. The following example proposes a way to choose $(a,b,\bar{\eta}_1, \bar{\eta}_2,
\gamma_1,\gamma_2, \beta )$ such that
\eqref{eq:stage-NCP-ex-dro} has a solution.

\begin{ex}\rm
Choose $(a,b,\bar{\eta}_1, \bar{\eta}_2,
\gamma_1,\gamma_2, \beta )$ with  $a> 2$ such that $z=\hat{\Pi}^{-1}\begin{pmatrix} a -\beta-1 \\ a -\beta-1 \end{pmatrix} $ satisfies
$0< z_i \leq \inf_{\xi\in \Xi} \tilde y_i(\xi)$,  for $ i=1,2$, where
$$
\hat{\Pi} := \begin{pmatrix}
2b + \bar{\eta}_1 -\gamma_1 & b \\
b & 2b + \bar{\eta}_2-\gamma_2\end{pmatrix}.
$$
Then we can show that
problem \eqref{eq:stage-NCP-ex-dro} has a  solution $(x_1, x_2, y_1(x_1, \cdot), y_2(x_2, \cdot), \mu_1(x, \cdot), \mu_2(x, \cdot), \bm{\Lambda})$ with
$
x_i = y_i(z_i, \xi) = z_i,
$
\begin{equation}\label{eq:mui}
\displaystyle
\mu_i(z, \xi)
= a + \xi_2  -  1 -(2b+\bar{\eta}_i+\xi_1) z_i - b z_{-i} \;\; \inmat{ and } \bm{\Lambda}_1= -\bm{\Lambda}_2 = \begin{pmatrix}
\lambda_{11} & \lambda_{12}\\
\lambda_{21} & \lambda_{22}
\end{pmatrix}
=
\begin{pmatrix}
z_1 & -1\\
z_2 & -1
\end{pmatrix}
\end{equation}
for all $\xi_1, \xi_2\in \Xi$ and $i=1,2$.
\label{p:ex}
\end{ex}

To see this, we  consider second stage complementarity problem of \eqref{eq:stage-NCP-ex-dro} (forth equation of \eqref{eq:stage-NCP-ex-dro}).
It is 
not difficult to verify that when
 $ x_i \leq \tilde y_i(\xi)$, for all $\xi\in \Xi$ and $i=1,2$, the forth
  equation of \eqref{eq:stage-NCP-ex-dro} has
 a solution $y_i(x_i, \xi) = x_i$, for $i=1,2$ with
$$
\begin{array}{lll}
\begin{pmatrix}
\mu_1(x, \xi)\\
\mu_2(x, \xi)
\end{pmatrix}
&=&
-\Pi(\xi_1)
\begin{pmatrix}
x_1\\
x_2
\end{pmatrix}
+
\begin{pmatrix}  a+\xi_2-1 \\ a+\xi_2-1 \end{pmatrix}
=
\Pi(\xi_1)(\tilde y(\xi) -x)\\
&=&
\begin{pmatrix}
a + \xi_2  -  1 -(2b+\bar{\eta}_i+\xi_1) x_1 - b x_{2}\\
a + \xi_2  -  1 -(2b+\bar{\eta}_i+\xi_1) x_1 - b x_{2}
\end{pmatrix}
\geq 0, \;\; \forall \xi\in \Xi.
\end{array} 
$$
In what follows,  we  show that  $(z_1, z_2, y_1(z_1, \xi), y_2(z_2, \xi), \mu_1(z, \xi), \mu_2(z, \xi) )$ satisfies the first 
equation of \eqref{eq:stage-NCP-ex-dro} with
$
\bm{\Lambda}_1=-\bm{\Lambda}_2 = \begin{pmatrix}
z_1 & -1\\
z_2 & -1
\end{pmatrix}.
$
With the explicit formulation of $\mu_i(x, \xi)$ from the second stage of  \eqref{eq:stage-NCP-ex-dro}, the first three equations in  \eqref{eq:stage-NCP-ex-dro} can be rewritten as
\bgeqn\label{eq:firststageDROex}{\small
\left\{
\begin{array}{ll}
x_1, x_2\geq0,\\
\left[-\begin{pmatrix}
\gamma_1 - (2b + \bar{\eta}_1 + \xi_1) & -b\\
-b & \gamma_2 - (2b + \bar{\eta}_2 + \xi_1)
\end{pmatrix}
\begin{pmatrix}
x_1\\
x_2
\end{pmatrix}
-
\begin{pmatrix}
a+\xi_2-\beta-1\\
a+\xi_2-\beta-1
\end{pmatrix} -
\bm{\Lambda}_1
\begin{pmatrix}
\xi_1\\
\xi_2
\end{pmatrix}\right]\leq0\\
\begin{pmatrix}
x_1\\
x_2
\end{pmatrix}^T
\left[\begin{pmatrix}
\gamma_1 - (2b + \bar{\eta}_1 + \xi_1) & -b\\
-b & \gamma_2 - (2b + \bar{\eta}_2 + \xi_1)
\end{pmatrix}
\begin{pmatrix}
x_1\\
x_2
\end{pmatrix}
+
\begin{pmatrix}
a+\xi_2-\beta-1\\
a+\xi_2-\beta-1
\end{pmatrix}
-\bm{\Lambda}_2
\begin{pmatrix}
\xi_1\\
\xi_2
\end{pmatrix}\right]\leq0
\end{array}
\right.}
\edeqn
for all $\xi\in\Xi$.
When $x_1, x_2>0$,
we can obtain a solution to
\eqref{eq:firststageDROex}
by solving
\bgeqn\label{eq:equationfirstex}
\left\{
\begin{array}{lll}
\hat{\Pi}\begin{pmatrix} x_1 \\ x_2 \end{pmatrix} = \begin{pmatrix} a -\beta-1 \\ a -\beta-1 \end{pmatrix},\\
\xi_2 -x_1\xi_1+\lambda_{11}\xi_1+\lambda_{12}\xi_2 =0,  & \forall \xi\in\Xi,\\
\xi_2 -x_2\xi_1+\lambda_{21}\xi_1+\lambda_{22}\xi_2 =0, & \forall \xi\in\Xi,
\end{array}
\right.
\edeqn
 whereby from the  first equation we
 have
 $x= \hat{\Pi}^{-1}\begin{pmatrix} a -\beta-1 \\ a -\beta-1 \end{pmatrix} =z$ and $\lambda_{i1}=z_i$,
    $\lambda_{i2}=-1$ for $i=1,2$ for the second equation.

It is worth noting that it is easy to choose $(a, b, \bar{\eta}_1, \bar{\eta}_2, \gamma_1, \gamma_2, \beta)$ such that $0< z_i \leq \inf_{\xi\in \Xi} \tilde y_i(\xi)$,  for $ i=1,2$. For example, when $\bar{\eta}_i\geq 1$,  $i=1,2$, we have
$$
\begin{array}{lll}
\begin{pmatrix}
\inf_{\xi\in \Xi} \tilde y_1(\xi)\\
\inf_{\xi\in \Xi} \tilde y_2(\xi)
\end{pmatrix}
&=& \Pi(1)^{-1}\begin{pmatrix}  a-2 \\ a-2 \end{pmatrix} = \displaystyle \frac{1}{\det(\Pi(1))}\begin{pmatrix}2b + \bar{\eta}_2+1 & -b \\
-b & 2b + \bar{\eta}_1+1  \end{pmatrix}\begin{pmatrix}  a-2 \\ a-2 \end{pmatrix}\\
& = & \displaystyle\frac{1}{\det(\Pi(1))} \begin{pmatrix}(b + \bar{\eta}_2+1)(a-2) \\
(b + \bar{\eta}_1+1)(a-2)  \end{pmatrix}
\end{array}
$$
and
$${\small
z=   \frac{1}{\det(\hat{\Pi})}\begin{pmatrix}2b + \bar{\eta}_2 - \gamma_2 & -b \\
-b & 2b + \bar{\eta}_1-\gamma_1  \end{pmatrix}\begin{pmatrix}  a-\beta-1 \\ a-\beta-1 \end{pmatrix}
 = \displaystyle\frac{1}{\det(\hat{\Pi})} \begin{pmatrix}(b + \bar{\eta}_2-\gamma_2)(a-\beta-1) \\
(b + \bar{\eta}_1-\gamma_1)(a-\beta-1)  \end{pmatrix}.}
$$
Moreover, when
\begin{equation}\label{eq:conditionex}
\frac{(a-2)(\bar \eta_{i}+1+b)}{(a-\beta-1)(\bar \eta_{i} + b -\gamma_{i})} \geq \frac{\det(\Pi(1))}{\det(\hat{\Pi})}, \;\; \inmat{ for } i=1,2,
\end{equation}
we have $z_i\leq \inf_{\xi\in \Xi} \tilde y_i(\xi)$, $i=1,2$.

\subsubsection{Numerical Tests}

We now move on to verify the discretization scheme discussed in Section 4
for \eqref{eq:stage-NCP-ex-dro}.
To this end, we set particular values for the underlying parameters in the table below:
\begin{center}
\begin{tabular}{|c |c c c c c c cccccccc| }
 \hline
 Parameters & $\alpha_1$ & $\beta_1$ & $\gamma_1$ & $s_1$ & $\zeta_1$ & $\bar \eta_1$ & $a$ &  $\alpha_2$ & $\beta_2$ & $\gamma_2$ & $s_2$ & $\zeta_2$ & $\bar \eta_2 $ & $b$ \\
 \hline
Values & 3 & 5 & 1 & 2 & 1 & 1 & 10 & 3 & 5 & 0.5 & 2 & 1 & 2 & 5 \\
  \hline
\end{tabular}
\end{center}
Let $\Xi_K:=\{\xi^1, \cdots, \xi^K\}$ be i.i.d. samples of $\xi$, where $\xi=(\xi_1, \xi_2)$,
$\xi_i$ follow truncated normal distribution over $[-1,1]$
 which is constructed from normal distribution with mean $0$ and
 standard deviation $\sigma$ independently, $i=1,2$.
 We carry out numerical experiments with  different values of $\sigma$. 
 With the specified parameter values,
condition \eqref{eq:conditionex}
is satisfied and 
consequently all conditions in Example \ref{p:ex} are fulfilled. 
This means 
we are able to obtain 
a solution of \eqref{eq:stage-NCP-ex-dro} with
$x_1  = y_1(\xi) = 0.2844, x_2 = y_2(\xi) = 0.2222$ and
$$
\mu_1(\xi) = 4.7111+\xi_2-0.2844\xi_1, \;\; \inmat{and} \;\; \mu_2(\xi) = 4.8889 +\xi_2 -0.2222\xi_1,  \;\;\forall \xi\in \Xi.
$$

Let us now apply the discrete 
scheme to \eqref{eq:stage-NCP-ex-dro}.
 The solution of the
 discretized problem is
$x^K_1  = ({\bf y}^K_{j})_1= 0.2844, x^K_2 = ({\bf y}^K_{j})_2 = 0.2222$ and
$$
\mu^K_1(\xi^j) = 4.7111+\xi^j_2-0.2844\xi^j_1, \;\; \inmat{and} \;\; \mu^K_2(\xi^j) = 4.8889 +\xi^j_2 -0.2222\xi^j_1,  \;\; \inmat{for} \;\;  j\in\bar{K}.
$$
Since $x^K=x$ and $y^K(\xi)=y(\xi)$ for all $\xi\in\Xi$, where $y^K(\xi) = \sum_{j=1}^K {\bf y}^K_j {\bf 1}_{\Xi_j^K}(\xi) $, $\Xi_j^K$ is defined by the  Voronoi
 partition in \eqref{eq:XiiK}. We only need to 
 investigate
 the error of $\mu^K(\xi^j)$, for $j\in\bar{K}$.
Define the errors of two components of  $\mu^K$
obtained
from the
discretized
problem by
$$
\inmat{error}_i^K = \bbe[|\mu_i(\xi) - \bar{\mu}_i^K(\xi)|],
\;\;
\inmat{where}
\;\;
\bar{\mu}^K(\xi) = \sum_{j=1}^K \mu^K(\xi^j) {\bf 1}_{\Xi_j^K}(\xi) , \;\; \inmat{for} \;\;  i=1,2,
$$
and the mathematical expectation is taken with respect to
the distribution of $\xi$ (truncated normal distribution).

  Note that it is not easy to calculate error$_i^K$ directly.
  Therefore we propose to use sample average approximation method to
  estimate  the quantity, that is, generate iid samples $\hat{\xi}^1,\cdots,\hat{\xi}^N$
  with sample size $N=5000$,
  and calculate
 $$
\inmat{error}^K_{i} \approx \inmat{error}^K_{iN} = \frac{1}{N}\sum_{k=1}^N|\mu_i(\hat{\xi}^k) - \bar{\mu}_i^K(\hat{\xi}^k)|, \;\;
\inmat{for}\; i=1,2.
$$
Here we are using notation $\hat{\xi}^k$ to distinguish the samples from those in $\Xi_K$.
We carried out tests with
sample sizes $K=5, 10, 20, 40, 60, 100$ and the standard deviation
$\sigma=0.1, 0.5, 1, 10$ of the normal distribution. For
each fixed
$K$ and $\sigma$,
 we generate
 $K$ samples $\Xi_K$, calculate the error$_{iN}^K$ $100$ times and  average  them.
 Figures \ref{fig:side:a}-\ref{fig:side:b} depict
 the decreasing tendency of error$_{iN}^K$ as $K$ increases and $\sigma$ decreases.
\begin{figure}[!htb]
\begin{minipage}[t]{0.5\textwidth}
\centering
\includegraphics[width=3.2in, height=2.2in]{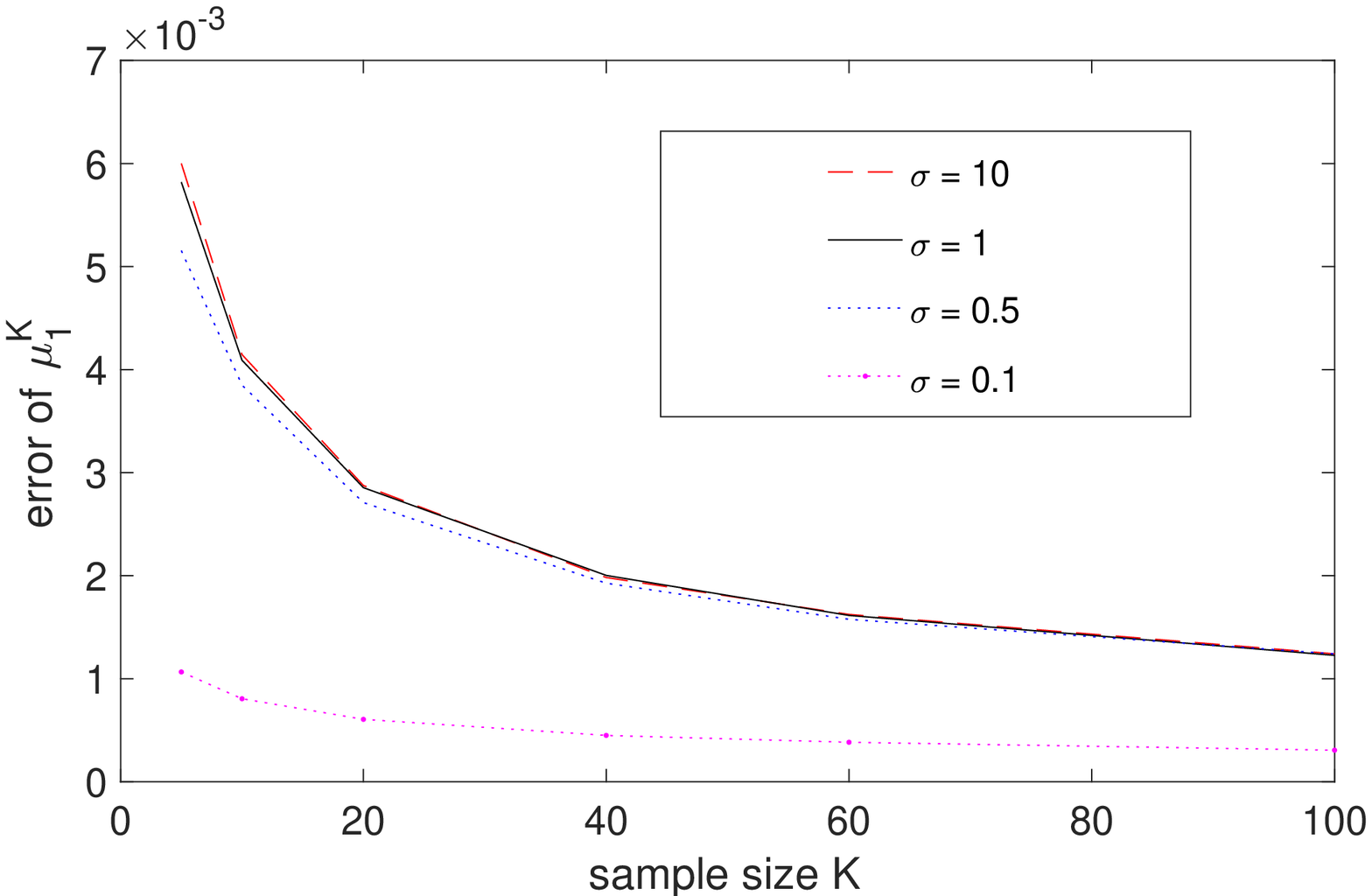}
\caption{error of $\mu_1^K$ }
\label{fig:side:a}
\end{minipage}%
\begin{minipage}[t]{0.5\textwidth}
\centering
\includegraphics[width=3.2in, height=2.2in]{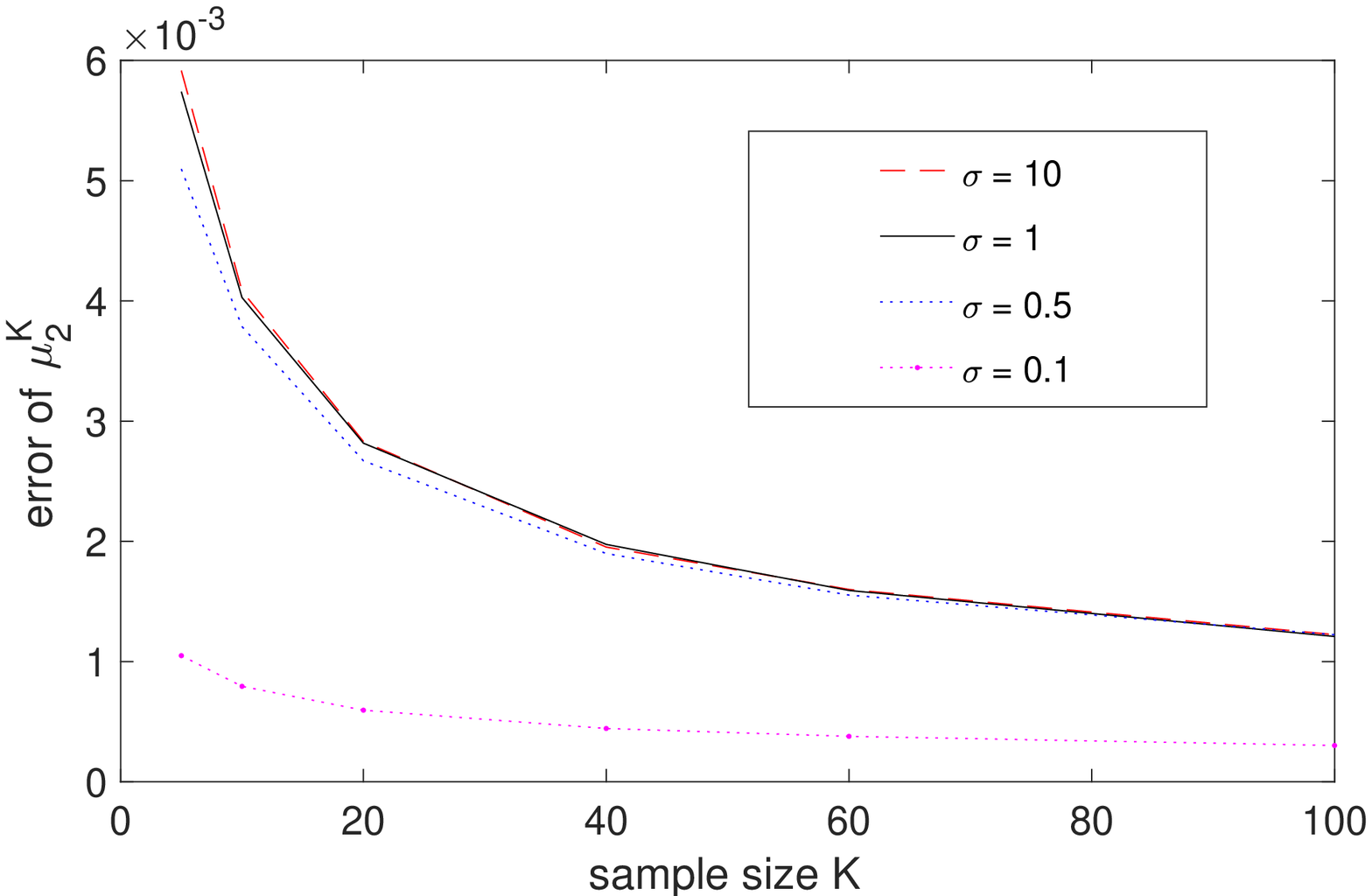}
\caption{error of $\mu_2^K$}
\label{fig:side:b}
\end{minipage}
\end{figure}

{\small

%\section{Appendix}

}

\end{document}